\documentclass[11pt]{amsart}
\usepackage{color}
\usepackage{amsmath,amscd,amssymb,amsfonts,amsthm}
\usepackage{tcolorbox}
\usepackage{tikz-cd}
\usepackage[cmtip,color,matrix,arrow]{xy}

\newtheorem{theorem}{Theorem}[section]

\newtheorem{proposition}[theorem]{Proposition}

\newtheorem{lemma}[theorem]{Lemma}
\theoremstyle{definition}    

\theoremstyle{remark}

\newtheorem{remark}[theorem]{Remark}

\newtheorem{example}[theorem]{Example}
\newtheorem{examples}[theorem]{Examples}

\newcommand{\sD}{\mathsf{D}}
\newcommand{\sC}{\mathsf{C}}

\newcommand\G{{G}}

\renewcommand{\L}{\mathcal{L}}

\newcommand{\ca}{\mathcal}

\newcommand{\F}{\mathcal{F}}

\newcommand{\R}{\mathbb{R}}

\newcommand\pt{\on{pt}}

\newcommand{\id}{\on{id}}
\newcommand\lie[1]{\mathfrak{#1}}
\renewcommand{\k}{\mathsf{k}}
\newcommand{\h}{\mathsf{h}}
\newcommand{\g}{\lie{g}}

\renewcommand{\a}{\mathsf{a}}

\newcommand{\on}{\operatorname}

\renewcommand{\ker}{ \on{ker}}

\newcommand{\VE}{\on{VE}}

\newcommand{\sz}{\mathsf{s}}
\newcommand{\tz}{\mathsf{t}}

\newcommand\qu{/\kern-.7ex/} 

\newcommand{\lra}{\longrightarrow}
\newcommand{\hra}{\hookrightarrow}

\renewcommand{\d}{{\mathsf{d}}}

\newcommand\eps{\epsilon}

\newcommand{\eeq}{\end{eqnarray*}}
\newcommand{\beq}{\begin{eqnarray*}}

\newcommand{\D}{\ca{D}}

\newcommand{\pr}{\on{pr}}
\newcommand{\wh}{\widehat}
\newcommand{\wt}{\widetilde}
\newcommand{\mf}{\mathfrak}
\newcommand{\rra}{\rightrightarrows}

\renewcommand{\subset}{\subseteq}
\newcommand{\btc}{\begin{tcolorbox}}
\newcommand{\etc}{\end{tcolorbox}}

\renewcommand{\sp}{\mathsf{p}}
\newcommand{\sq}{\mathsf{q}}
\newcommand{\si}{{\mathsf{i}}}
\newcommand{\sj}{{\mathsf{j}}}
\newcommand{\Ra}{\Rightarrow}
\newcommand{\sX}{\mathsf{X}}
\newcommand{\sY}{\mathsf{Y}}

\begin{document}
\title{Van Est differentiation and integration}

\author{Eckhard Meinrenken}\address{Department of Mathematics, University of Toronto (Canada)}\email{mein@math.toronto.edu}\author{Maria Amelia Salazar}\address{Instituto de Matematica e Estatistica,  Universidade Federal Fluminense (Brazil)}\email{mariasalazar@id.uff.br}

\begin{abstract}
The classical Van Est theory relates the smooth cohomology of Lie groups with the cohomology of the associated Lie algebra, or its relative versions. Some aspects 
of this theory generalize to Lie groupoids and their Lie algebroids. In this paper, continuing an idea from \cite{lib:ve}, we revisit the van Est theory using the \emph{Perturbation Lemma} from homological algebra. Using this technique, we obtain precise results for  the van Est differentiation and integrations maps \emph{at the level of cochains}. Specifically, we construct 
homotopy inverses to the van Est differentiation maps that are right inverses at the cochain level. 
\end{abstract}
\maketitle
\tableofcontents

\section{Introduction}
In a series of papers \cite{vanest:gro,vanest:alg,vanest:appl} in the early 1950s, Willem van Est 
established several key facts relating the smooth group cohomology  of a Lie group $G$ to the cohomology of its associated Lie algebra $\g$. One of his results describes a cochain map $\VE_G$ from the Lie group complex to the Chevalley-Eilenberg Lie algebra complex, which induces an isomorphism in cohomology  up to a certain degree depending on the connectivity properties of $G$. (Using a localized complex, working with germs near the group unit, it induces  an isomorphism in all degrees \cite{hou:int,swi:coh}; see also \cite{hu:coh}.) Furthermore, van Est proved that  the smooth group cohomology of a connected Lie group $G$ is canonically isomorphic to the relative Lie algebra cohomology of $\g$ with respect to the maximal compact subgroup $K$ on $G$. An explicit cochain map from the relative Lie algebra complex $\sC^\bullet(\g,K)=\sC^\bullet(\g)_{K-\on{basic}}$ to the complex  $\sC^\bullet(G)$ was described later by Dupont \cite{dup:sim}, Shulman-Tischler \cite{shu:lea}, and Guichardet \cite{gui:coh}; see \cite{dup:pro,gui:sur,hou:int,swi:coh} for  applications and generalizations. The van Est map was extended by Weinstein-Xu \cite{wei:ext} to Lie groupoids $G\rra M$, as a cochain map $\VE_G\colon \sC^\bullet(G)\to \sC^\bullet(A)$
from the smooth groupoid cochain complex to the Chevalley-Eilenberg complex of its Lie algebroid $A=\on{Lie}(G)$. Versions of the van Est theorems for Lie groupoids were obtained by Crainic \cite{cra:dif}. More recently, an explicit homotopy inverse
\[ R_G\colon \sC^\bullet(A)\to \sC^\bullet(G)_M\]
(where the subscript indicates the localized complex)
was found by Cabrera-Marcut-Salazar \cite{cab:loc}. 

In this article, we will revisit the van Est theory using the \emph{Perturbation Lemma} from homological algebra. For the map $\VE_G$, this was initiated by Li-Bland and Meinrenken in \cite{lib:ve}, but we will show that it carries much further. In short, this approach constructs the cochain maps in the van Est theory \emph{systematically},  from homotopy operators on various double complexes (as opposed to `guessing' the right formulas). The properties of these cochain maps are 
obtained from properties of these homotopy operators. This leads to a number of  observations that were  missed in earlier literature. All our results 
apply to cochain groups with coefficients in a given $G$-representation $V$, but for simplicity we will only describe 
the scalar case in the following summary:\medskip

{\bf Van Est theory for Lie groups.} We begin by revisiting 
the classical setting that $G$ is a Lie group, and  $K$ a compact Lie subgroup. We describe a distinguished horizontal homotopy operator on the  van Est double complex, and use it to obtain a canonical van Est differentiation map 
\[ \VE_{G/K}\colon \sC^\bullet(G)\to \sC^\bullet(\g)_{K-\on{basic}},\]
with values in the  relative Lie algebra complex. We will show that this relative 
van Est map is the composition 
\begin{equation}\label{eq:vek1}
\VE_{G/K}=\VE_G\circ \on{Av}\end{equation}
where $\on{Av}\colon \sC(G)\to \sC(G)$ is  given on degree $p$ elements by averaging under a natural $K^{p+1}$-action. 
If $G$ has finitely many components, and $K$ is a \emph{maximal} compact subgroup, then the diffeomorphism $G/K\cong \g/\mf{k}$ determines a vertical homotopy operator on the double complex, and a resulting cochain map (`integration')
\[ R_{G/K}\colon \sC^\bullet(\g)_{K-\on{basic}}\to \sC^\bullet(G).\]
This map is similar to (but not equal to) the  map defined in \cite{dup:sim,gui:coh,shu:lea}. 
Our theory shows that 
\begin{equation}\label{eq:ver}
\VE_{G/K}\circ R_{G/K}=\id
\end{equation} 
\emph{at the level of cochains}. Equations \eqref{eq:vek1} and \eqref{eq:ver} provide a strengthening of van Est's original results, which are stated at the level of cohomology. 
For arbitrary compact subgroups $K$ of $G$ (not necessarly maximal compact ones), we have a similar statement for the localized complex; in particular, this applies to $K=\{e\}$. 
\medskip

{\bf Van Est maps for Lie groupoids.} For a Lie groupoid $G\rra M$, with Lie algebroid $A=\on{Lie}(G)$,  it was shown in \cite{lib:ve} how recover the van Est differentiation map $\VE_G$ of \cite{wei:ext}, through applications of the Perturbation Lemma to the van Est double complex $\sD^{\bullet,\bullet}(G)$ from \cite{cra:dif}. Given a (germ of a) `tubular structure' for $G$, we also have a vertical homotopy $\k$ on the double complex. We will prove that 
the resulting van Est integration map $R_G\colon \sC(A)\to \sC(G)_M$, with values in the localized complex,  coincides with the integration map of  \cite{cab:loc}. The fact that $\VE_G,\,R_G$ are cochain maps, and that 
\[ \VE_G\circ R_G=\id\]
 on $\sC(A)$, are obtained as immediate consequences 
of the properties of $\h,\k$ and a general algebraic lemma, avoiding the calculations in 
\cite{wei:ext} and \cite{cab:loc}.\medskip

{\bf  Van Est maps for Lie groupoid actions.} Here we consider groupoid actions of $G\rra M$ on manifolds $Q$, with anchor map $\Phi\colon Q\to M$ a surjective submersion. The Lie algebroid complex $\sC(A)$  is generalized to a foliated de Rham complex $\Omega_\F(Q)^G$ of invariant leafwise forms along the fibers of $\Phi$. (For $Q=G$ with the left action, one recovers $\sC(A)$.) 
According to Crainic \cite{cra:dif}, if the action is proper, then the choice of a suitable `Haar distribution' on the action groupoid gives a horizontal homotopy operator $\h$ on a double complex $\sD^{\bullet,\bullet}(Q)$. In turn, using the Perturbation Lemma, this determines a differentiation map $\VE_Q$ from $\sC(G)$ to $\Omega_\F(Q)^G$. On the other hand, given a right inverse 
$M\hra Q$ to $\Phi$ and a tubular neighborhood embedding, one obtains a vertical homotopy $\k$ and hence an integration map $R_Q$ in the opposite direction, 
at least after localizing. We characterize situations where this integration map is right inverse to differentiation. 
\medskip

 Each of the three themes outlined above constitutes a section of this article; these sections are preceded by a quick review of the Perturbation Lemma, which will be our main tool throughout the paper.

\subsection*{Acknowledgments}  E.M. was supported by an NSERC Discovery Grant. This study was financed in part by the Coordena\c{c}\~ao de Aperfei\c{c}oamento de Pessoal de N\'ivel Superior - Brasil (CAPES) - Finance code 001. The authors would like to thank the hospitality of Fields Institute where some of this research was carried out.

\section{The Perturbation Lemma}\label{sec:perturbationlemma}
Let 
$(\sD^{\bullet,\bullet},\d,\delta)$ be a double complex, concentrated in non-negative degrees,

\[ \xymatrix{ & & \\    
	  \sD^{0,1} \ar[r]_\delta \ar[u]^d & \sD^{1,1}\ar[r]_\delta\ar[u]^d &           \\  \sD^{0,0} \ar[r]_\delta\ar[u]^d & \sD^{1,0}\ar[r]_\delta\ar[u]^d& \\
}\]
Let $(\sX^{\bullet},\d)$ be a cochain complex. A morphism of double complexes
\[ \si\colon \sX\to  \sD\]
 (where $\sX^\bullet$ is regarded as a double complex concentrated in bidegrees $(0,\bullet)$) will be called a \emph{horizontal augmentation map}. 
Passing to total complexes, $\si$ becomes a  cochain map from $X^\bullet$ to the cochain complex $(\on{Tot}^\bullet(\sD),\d+\delta)$. The Perturbation Lemma, due to Brown \cite{bro:twi} and  Gugenheim \cite{gug:cha}, 
allows us to turn a homotopy operator for the horizontal differential $\delta$ into a homotopy operator with respect to the total differential $\d+\delta$. 
%
\begin{lemma}[Perturbation Lemma] \label{lem:perturbed}
		Suppose $\mathsf{h}\colon \sD\to \sD$ is a linear map of bidegree  $(-1,0)$, such that 
		\[ [\mathsf{h},\delta]=1-\si\circ \sp\] for some degree $0$ map 
		$\sp\colon \sD^{0,\bullet}\to \sX^\bullet$.
		Put 
		$\mathsf{h}'=\mathsf{h}(1+\d \mathsf{h})^{-1}$ and $ \sp'=\sp(1+\d \mathsf{h})^{-1}$. 
		Then 
		\[ [\mathsf{h}',\d+\delta]=1-\si\circ\sp'.\]
\end{lemma}
Here, $[\cdot,\cdot]$ denotes the \emph{graded} commutator, e.g. $[\h,\delta]=\h\delta+\delta\h$.  
\begin{proof}
Using $\h'(1+\d \h)=\h=(1+\h\d)\h'$ one finds, by straightforward calculation, 
\[ (1+\mathsf{h}\d) [\mathsf{h}',\d+\delta]\,(1+\d \mathsf{h})
=[\h,\d+\delta]=[\h,\d]+1-\si\circ \sp.\]
Expanding $(1+\mathsf{h}\d) (1-\si  \circ\sp')
 (1+\d \mathsf{h})$, using $\h\d\si=\h\si\d=0$, gives the same result. 
\end{proof}	
	
We shall assume from now on that $\sp\circ \si=\id_{\sX}$, so that $\si$ is injective and $\si\circ \sp$ is a projection onto the image of $\si$. Then also $\sp'\circ \si=\id_{\sX}$, and $\si\circ \sp'$ is again a projection. 
In other words, $\si\colon \sX^\bullet\to \on{Tot}^\bullet(\sD)$ is a homotopy equivalence, with $\sp'$ a homotopy inverse. 

In our applications, there is another cochain complex $(\sY^\bullet,\delta)$, with a
\emph{vertical augmentation map} 
\[ \sj\colon \sY\to \sD\]
 (thus $\sY^\bullet$ is regarded as a double complex concentrated in bidegrees $(\bullet,0)$). The horizontal homotopy $\h$ allows us to `invert' the second cochain map in 
\[ \sY^\bullet\stackrel{\sj}{\lra} \on{Tot}^\bullet(\sD)\stackrel{\si}{\longleftarrow}\sX^\bullet,\]
thereby producing a cochain map 
$\sp'\circ \sj=\sp\circ (1+\d \h)^{-1}\circ \sj\colon \sY^\bullet\to \sX^\bullet$. 
On elements of degree $p$, this is given by a `zig-zag' 
\begin{equation}\label{eq:xy} (-1)^p \sp \circ (\d \h)^p \circ \sj\colon \sY^p\to \sX^p,\end{equation}
illustrated here for $p=2$:
\[ \xymatrix{ 
		\sX^2& \sD^{0,2}  \ar[l]_{\sp} &
	&  &            \\  & \sD^{0,1}\ar[u]^{\mathsf{d}} & \sD^{1,1}\ar[l]^{\mathsf{h}}&  &            \\ 
	 &  & \sD^{1,0}\ar[u]^{\mathsf{d}}& 
\sD^{2,0}\ar[l]^{\mathsf{h}}& 
	 \\ & & 
	 & 
	\sY^2
	\ar[u]^{\sj}&
}\]

\begin{example}\label{ex:chechderham}
Let $M$ be a manifold with a covering $\ca{U}=\{U_i\}$ by open sets, and let $\sD^{p,q}=\sC^p(\ca{U},\Omega^q)$ be the \v{C}ech-de Rham double complex. It comes 
with a horizontal augmentation map $\si\colon \Omega^q(M)\to \sD^{0,q}$ from the de Rham complex, 
and a vertical augmentation map 
$\sj\colon\sC^p(\ca{U},\underline{\R}) \to \sD^{p,0}$ from the \v{C}ech complex. Given a locally finite partition of unity $\{\chi_i\}$ subordinate to the cover, 
one obtains a horizontal homotopy operator $\h$, with $\sp$ the map taking a collection of $q$-forms 
$\omega_i\in \Omega^q(U_i)$ on the open sets to a global $q$-form $\sum_i \chi_i \omega_i\in \Omega^q(M)$. See Bott-Tu \cite[Proposition 8.5]{bo:di}. The resulting 
zig-zag  \eqref{eq:xy} defines a \v{C}ech-de Rham cochain map $\sC^p(\ca{U},\underline{\R})\to 
\Omega^p(M)$, which is nothing but the `collating formula' of Bott-Tu, \cite[Proposition 9.5]{bo:di}. 
\end{example}

Consider now the situation that the vertical differential has a homotopy operator 
\[ \k\colon \sD^{p,q}\to \sD^{p,q-1},\ \ [\d,\k]=1-\sj\circ \sq,\]
where 
$\sq\colon \sD^{0,\bullet}\to \sY^\bullet$ is a cochain map for $\d$ with $\sq\circ \sj=\id_{\sY}$. Then we can apply the Perturbation Lemma \ref{lem:perturbed} to 
this vertical homotopy, and we obtain a cochain map $\sq\circ (1+\delta\k)^{-1}\circ \si\colon \sX^\bullet\to \sY^\bullet$ given on degree $p$ elements by a zig-zag, 
\begin{equation}\label{eq:yx}  (-1)^p \sq \circ (\delta\k )^p \circ \si\colon \sX^p\to \sY^p.
\end{equation}
Note that the route taken by the zig-zag \eqref{eq:yx} retraces the steps of the zig-zag \eqref{eq:xy}.
The following result will be used to relate van Est `integration' and `differentiation'  maps. 
\begin{lemma}[Zig-zag back-and-forth]\label{lem:backandforth}
Suppose the homotopy operators $\h,\k$ 
satisfy 
\begin{equation}\label{eq:hk0} \h\circ \k=0,\ \ \ \ \sp\circ \k=0.\end{equation}
Then \eqref{eq:yx} followed by 
\eqref{eq:xy} is the identity map of $\sX^p$.
\end{lemma}
\begin{proof}
We first note that 	
\begin{equation}\label{eq:4} \sp\circ  \sj \circ \sq\big|_{\sD^{0,0}}=\sp\big|_{\sD^{0,0}},\end{equation}
and for $p>0$, 	
\begin{equation} \label{eq:3}\h\circ  \sj \circ \sq\big|_{\sD^{p,0}}=\h\big|_{\sD^{p,0}}.\end{equation} 
Equation \eqref{eq:3} follows from the calculation, for $p>0$, 	
\[ \h \circ (1-\sj\circ \sq) \big|_{\sD^{p,0}}=\h\circ [\d,\k]|_{\sD^{p,0}}=\h\circ \d\circ \k|_{\sD^{p,0}}=0\]
where we used that $\k$ vanishes on $\sD^{p,0}$ for degree reasons. Equation \eqref{eq:4} is obtained similarly. 
The Lemma now follows for $p=0$ from 
\[ \sp\circ \sj\circ \sq\circ \si|_{\sX^0}=\sp\circ \si|_{\sX^0}=\id_{\sX^0},\]
using \eqref{eq:4}, and for $p>0$ from the following calculation, as operators on $\sX^p$:
\begin{align*}
\sp\circ (\d \h)^p \circ \sj \circ \sq \circ (\delta \k)^p \circ \si&=
\sp\circ (\d \mathsf{h})^p \circ (\delta \k)^p \circ \si\\
 &=\sp\circ (\d \mathsf{h})^{p-1} \circ 
(1-\k\d )
 \circ (\delta \k)^{p-1} \circ \si\\&=\sp\circ(\d \mathsf{h})^{(p-1)}(\delta \k)^{p-1}\circ \si\\
 &=\ldots
 \\ &=\sp\circ \si\\
 &=\id_{\sX^p}.
\end{align*}
In the first equality, we used \eqref{eq:3} if $p>0$, or \eqref{eq:4} if $p=0$, 
to omit the $\sj\circ \sq$ factor. The second equality follows from 
\[ \d\mathsf{h}\delta\k\big|_{\sD^{p,0}}=\d(1-\si\sp-\delta \mathsf{h})\k\big|_{\sD^{p,0}}=\d\k\big|_{\sD^{p,0}}=\id_{\sD^{p,0}}-\k\d\big|_{\sD^{p,0}}.\] 
Here we used $\mathsf{h} \k=0$ and $\sp \k=0$. 
\end{proof}

\begin{remark}
In the \v{C}ech-de Rham example \ref{ex:chechderham}, suppose that the cover $\ca{U}$ is a \emph{good} cover, so that all non-empty intersections of the $U_i$ are contractible. Then  the choice of such retractions defines a vertical homotopy operator $\k$, and hence gives a cochain map $\Omega^\bullet(M)\to \sC^\bullet(\ca{U},\underline{\R})$ in the opposite direction. Unfortunately, the conditions of Lemma 
\ref{lem:backandforth} are \emph{not} satisfied, in general; hence this map won't give a right inverse to the \v{C}ech-de Rham cochain map, even though it is a homotopy inverse. 
\end{remark}

\section{Van Est theory for Lie groups}\label{sec:classical}
Suppose that $G$ is a Lie group with finitely many connected components. One of van Est's results, often referred to as the \emph{van Est theorem}, is that the smooth cohomology of $G$ with coefficients in a representation $V$ 
is canonically isomorphic to the relative Lie algebra cohomology with respect to a maximal compact subgroup $K\subset G$, with coefficients in $V$.  As we will see, the Perturbation Lemma will guide us towards explicit van Est maps, in both directions. Furthermore, we will show that the `integration map' is a right inverse to the `differentiation map', at the level of cochains. \medskip

\subsection{The van Est double complex}
Let $G$ be a Lie group, with a representation on a vector space $V$. The smooth Lie group cochain complex $(\sC^\bullet(G,V),\delta)$ has as its $p$-cochains the functions
\[ \sC^p(G,V) =C^\infty(G^p,V),\]
and the differential is given by 
\begin{align} \label{eq:deltaformula}
(\delta f)(g_1,\ldots,g_{p+1})&=f(g_2,\ldots,g_{p+1})+
\sum_{i=1}^{p}(-1)^i 
f(g_1,\ldots,g_ig_{i+1},\ldots,g_{p+1})\\
&\ \ +(-1)^{p+1} (g_{p+1})^{-1}\cdot f(g_1,\ldots,g_p).\nonumber \end{align}
On the other hand, letting $\g=\on{Lie}(G)$, we have the usual 
Lie algebra complex $(\sC^\bullet(\g,V),\d_{CE})$, where  
\[ \sC^q(\g,V)=V\otimes \wedge^q \g^*,\]
and where $\d_{CE}$ is the Chevalley-Eilenberg differential.  The
Lie group complex and Lie algebra complex are related by a double complex $(\sD^{\bullet,\bullet}(G,V),\delta,\d)$ introduced in van Est's original articles \cite{vanest:gro,vanest:alg,vanest:appl}; see also Guichardet \cite{gui:coh}.
The van Est double complex has bigraded components 
\begin{equation}\label{eq:dpq} \sD^{p,q}(G,V)=C^\infty(G^p\times G, V\otimes \wedge^q\g^*),
\end{equation}
and the horizontal differential  is given by 
\begin{align*}{(\delta\psi)(g_1,\ldots,g_{p+1};g)}&=\psi(g_2,\ldots,g_{p+1};g)+
\sum_{i=1}^p (-1)^i \psi(g_1,\ldots,g_ig_{i+1},\ldots,g_{p+1};g)
\\&\ \ +(-1)^{p+1}\psi(g_1,\ldots,g_p;g_{p+1}g).\end{align*}
The vertical differential is given by $\d=(-1)^p\d_{CE}$, where we identify $\sD^{p,\bullet}(G,V)$ with the Chevalley-Eilenberg complex of the $\g$-representation on $C^\infty(G^p\times G)\otimes V$, using the infinitesimal $\g$-representation on $V$ and the representation $\xi\mapsto (0,\xi^L)$ (the left-invariant vector field on the last $G$-factor) on $C^\infty(G^p\times G)$. Then $[\d,\delta]=\d\delta+\delta\d=0$, as desired.  
This double complex has horizontal and vertical augmentation maps
\[ \si\colon \sC^{\bullet}(\g,V)\to \sD^{0,\bullet}(G,V),\ \ \ 
\sj\colon \sC^\bullet(G,V)\to  \sD^{\bullet,0}(G,V),\]
where $\si$ is the inclusion of constant functions, while
\[ (\sj f)(g_1,\ldots,g_p;g)=
g^{-1}\cdot f(g_1,\ldots,g_p)\]
for $f\in \sC^p(G,V)=C^\infty(G^p,V)$. Note $\sj$ is an inclusion of $\sC(G,V)$ as 
the $G$-invariant part of $\sD^{\bullet,0}(G,V)$, or equivalently the $G$-basic part of the full double complex, 
with respect to the following $G$-action 
\begin{equation}\label{eq:principalaction}
(a\cdot\psi)(g_1,\ldots,g_p;g)=a\cdot \psi(g_1,\ldots,g_p;ga)
\end{equation}
(using the coadjoint action on $\wedge\g^*$).

Now suppose $K\subset G$ is a compact Lie subgroup. The \emph{relative Lie algebra complex with coefficients in $V$} is 
the $K$-basic subcomplex 
\[ \sC(\g,V)_{K-\on{basic}}.\]
That is, it consists of elements that are annilhilated by contractions with elements of $\mf{k}$ on the $\wedge\g^*$-factor, and are invariant for the action of $K\subset G$. Similarly, 
we can consider the 
$K$-basic sub-double complex $\sD(G,V)_{K-\on{basic}}$ with respect to the action \eqref{eq:principalaction}. The two augmentation maps for $\sD(G,V)$ restrict to augmentation maps 
\[ \si\colon \sC^{\bullet}(\g,V)_{K-\on{basic}}\to \sD^{0,\bullet}(G,V)_{K-\on{basic}},\ \ \ 
\sj\colon \sC^\bullet(G,V)\to \sD^{\bullet,0}(G,V)_{K-\on{basic}}.\]

\subsection{Differentiation}
The double complex $\sD(G,V)_{K-\on{basic}}$ has a horizontal homotopy operator:
\begin{equation}\label{eq:hmap} (\h\psi)(g_1,\ldots,g_{p-1};g)=(-1)^p \int_K\ \psi(g_1,\ldots,g_{p-1},gk^{-1};k)\d k;\end{equation}
here $\d k$ is the normalized invariant Haar measure on $K$. Indeed, a direct calculation shows that 
$\delta\h+\h\delta=1-\si\circ \sp$ where $\sp$ vanishes on elements of bidegree $(p,q)$ with $p>0$, while
\begin{equation}\label{eq:pmap}
 \sp(\psi)=\int_K \psi(k)\d k
,\end{equation} 
for $\psi\in \sD^{0,q}(G,V)_{K-\on{basic}}=C^\infty(G,V\otimes \wedge^q\g^*)_{K-\on{basic}}$.
\begin{remark}
For $K=\{e\}$, the homotopy operator simplifies to \[ (\h\psi)(g_1,\ldots,g_{p-1};g)=(-1)^p  \psi(g_1,\ldots,g_{p-1},g;e).\] Note that this homotopy operator on $\sD(G,V)$ does \emph{not} preserve the basic subcomplex with respect to a nontrivial compact subgroup. 
\end{remark}

Using the Perturbation Lemma \ref{lem:perturbed}, we obtain a cochain map (van Est differentiation)
\[ \VE_{G/K}=\sp\circ (1+\d\h)^{-1}\circ \sj\colon 
\sC^\bullet(G,V)\to \sC^\bullet(\g,V)_{K-\on{basic}}.\]
For an explicit description of this map, 
consider the following $p+1$ commuting
$G$-actions on the space $C^\infty(G^p,V)$, 
\begin{equation}\label{eq:agg}
(a\cdot f)(g_1,\ldots,g_p)=\begin{cases}
f(a^{-1}g_1,\ldots, g_p)& i=0,\\
f(g_1,\ldots,g_i\,a,a^{-1}g_{i+1},\ldots, g_p)
& 0< i <p,
\\ a\cdot f(g_1,\ldots,g_pa)& i=p.
\end{cases}\end{equation}
For each of these actions we can consider its restriction to $K$; 
denote by 
\[ \on{Av}^{(i)}\colon C^\infty(G^p,V)\to C^\infty(G^p,V)\]
the averaging operation with respect to $i$-th $K$-action. The operators $\on{Av}^{(i)}$ commute since the actions commute,  and we denote by  $\on{Av}=\on{Av}^{(0)}\circ \cdots \circ \on{Av}^{(p)}$ the total $K^{p+1}$-averaging operation. We will also need the $G$-actions, obtained from the actions \eqref{eq:agg} labeled $i,i+1,\ldots,p$ by passing to the diagonal action:
\begin{equation}\label{eq:tildeaction}
(a\cdot f)(g_1,\ldots,g_p)=a\cdot f(g_1,\ldots,g_i a,a^{-1}g_{i+1}a,\ldots,a^{-1}g_p a).    \end{equation} 
Denote by ${\d}^{(i)}$ the corresponding Chevalley-Eilenberg differential on $C^\infty(G^p,V)\otimes \wedge \g^*$. \medskip

\begin{theorem}\label{th:K}
	For any compact Lie subgroup $K\subset G$, the van Est  differentiation $\VE_{G/K}=\sp\circ (1+\d\h)^{-1}\circ \sj$ 
	is given by the formula 
	\begin{equation}\label{eq:vek}  \VE_{G/K}(f)=
	\big(\d^{(1)}\cdots \d^{(p)}\ \on{Av}(f)\big)\Big|_{(e,\ldots,e)},\end{equation}
	for $f\in C^\infty(G^p,V)$.
	In particular, $\VE_{G/K}=\VE_G\circ \on{Av}$. 
\end{theorem}\medskip
\begin{proof}
	The $G$-actions on $C^\infty(G^p,V)$, given by Equation \eqref{eq:agg}, extend to  commuting actions on 
	$C^\infty(G^p,V\otimes \wedge^q \g^*)$, given by the same formulas (replacing $V$ with $V\otimes \wedge^q\g^*$). 	Denote by $\on{Av}^{(i)}$ the averaging operation 
	on $C^\infty(G^p,V\otimes \wedge^q \g^*)$
	for the $i$-th action of $K\subset G$.	
	
	We calculate
	$\on{VE}_{G/K}=(-1)^p \sp\circ (\d\h)^{p}\circ \sj$ on  
	$f\in \sC^p(G,V)$. The first few steps are 
	\begin{align*}
	(-1)^p (\sj f)(g_1,\ldots,g_p;g)&=(-1)^p g^{-1}\cdot f(g_1,\ldots,g_p)\\
	(-1)^p (\h \sj f)(g_1,\ldots,g_{p-1};g)&=\int_K k^{-1}\cdot f(g_1,\ldots,g_{p-1},g k^{-1})\ \d k\\
	&= (\on{Av}^{(p)} f)(g_1,\ldots,g_{p-1},g)\\
	(-1)^p (\d\h\sj f)(g_1,\ldots,g_{p-1};g)&=(-1)^{p-1} (\d^{(p)} \on{Av}^{(p)} f)(g_1,\ldots,g_{p-1},g).
	\end{align*}
	In the last line, we used that the $G$-action \eqref{eq:principalaction} defining $\d_{CE}\psi$ for 
	\[ \psi(g_1,\ldots,g_{p-1};g)=(\on{Av}^{(p)} f)(g_1,\ldots,g_{p-1},g)\]
	corresponds to the $p$-th $G$-action \eqref{eq:agg} (after setting $g_p=g$), defining $\d^{(p)}$. Next, 	
	\begin{align*}
	(-1)^p (\h\d\h\sj f)( g_1,\ldots,g_{p-2};g)&=
	\int_K  (\d^{(p)} \on{Av}^{(p)} f)(g_1,\ldots,g_{p-2},gk^{-1},k)\ \d k\\
	&=(\on{Av}^{(p-1)}\d^{(p)} \on{Av}^{(p)} f)(g_1,\ldots,g_{p-2},g,e)\\
	(-1)^p (\d\h\d\h\sj f)( g_1,\ldots,g_{p-2};g)&=({\d}^{(p-1)}\on{Av}^{(p-1)}\d^{(p)} \on{Av}^{(p)} f)(g_1,\ldots,g_{p-2},g,e).
\end{align*}
Here we used that the $G$-action \eqref{eq:principalaction} defining $\d_{CE}\psi$ for 
\[ \psi(g_1,\ldots,g_{p-2};g)=(\on{Av}^{(p-1)}\d^{(p)} \on{Av}^{(p)} f)(g_1,\ldots,g_{p-2},g,e)\]
corresponds (for $g_{p-1}=g,\ g_p=e$) 
to the $(p-1)$-st diagonal $G$-action 
\eqref{eq:tildeaction}, defining ${\d}^{(p-1)}$. Continuing in this fashion, we arrive at 
	\[ (-1)^p\, \sp\circ (\d \h)^p\circ \sj\ f=(\on{Av}^{(0)}{\d}^{(1)}\on{Av}^{(1)}\cdots {\d}^{(p)}\on{Av}^{(p)} f)(e,\cdots,e).\]
	Since the $i$-th-action \eqref{eq:agg} commutes with the $j$-th action 
	\eqref{eq:tildeaction} for $i<j$, the operator 	
	$\on{Av}^{(i)}$ commutes with ${\d}^{(j)}$ for $i<j$. Hence  we may move the averaging operations all the way to the right, resulting in the formula \eqref{eq:vek}. 
\end{proof}

\begin{remark}
While there are various more or less explicit descriptions of the van Est differentiation  (see in particular \cite{gui:coh}), we are not 
aware of an appearance of the formula \eqref{eq:vek} in the literature, for general compact $K$. Note also that it is not necessary to pass to a `normalized subcomplex'.
\end{remark}

\subsection{Integration}
For the discussion of van Est integration maps, another interpretation of the relative (double) complex will be convenient. Let $G$ act on itself by left translation $g\mapsto ag$, 
and let $\Omega^q(G,V)^G$ be the corresponding complex of $G$-invariant $V$-valued forms, where the action of $V$ is the given representation.  Restriction of such a form to the group 
unit gives an isomorphism $\Omega^q(G,V)^G\to V\otimes \wedge^q \g^*$, which intertwines the 
de Rham differential and the Chevalley-Eilenberg differential. Thus 
\[ \sC^q(\g,V)\cong \Omega^q(G,V)^G\]
as differential complexes. The $G$-action on $\sC^q(\g,V)=V\otimes \wedge^q \g^*$ (with the given $G$-representation on $V$ and the coadjoint action on $\wedge\g^*$)
corresponds to the action 
on $ \Omega^q(G,V)^G$ coming from the action $g\mapsto ga^{-1}$ on $G$ and the trivial action  on $V$. Consider the restriction of this action to $K$; 
Since 
$\Omega^q(G,V)_{K-\on{basic}}=
\Omega^q(G/K,V)$, and similarly for the invariant forms for the action by left multiplication, 
we obtain the identification
\begin{equation}\label{eq:invforms} \sC^q(\g,V)_{K-\on{basic}}\cong 
\Omega^q(G/K,V)^G.\end{equation}
Similarly, elements  $\psi\in \sD^{p,q}(G,V)_{K-\on{basic}}$ may be identified with functions  
\begin{equation}\label{eq:betamap} \beta\colon G^p\to \Omega^q(G/K,V),\end{equation}
 with smooth dependence on $(g_1,\ldots,g_p)\in G^p$ as parameters. In these terms, the two
differentials are 
\begin{equation}\label{eq:dbeta}
 (\d\beta)(g_1,\ldots,g_p)=(-1)^p \d_{Rh} \beta(g_1,\ldots,g_p),
 \end{equation}
where $\d_{Rh}$ is the de Rham differential, and 
\begin{align} \label{eq:deltabeta}(\delta\beta)(g_1,\ldots,g_{p+1})&=
\beta(g_2,\ldots,g_{p+1})+\sum_{i=1}^p (-1)^i \beta(g_1,\ldots,g_ig_{i+1},\ldots,g_{p+1})
\\&\ \ +(-1)^{p+1} L(g_{p+1})^*\beta(g_1,\ldots,g_p)\nonumber\end{align}
where $L(a)\colon G/K\to G/K$ is 
the action of $a\in G$. (For $p=0$, this is to be interpreted as 
$(\delta\beta)(g_1)=\beta-L(g_1)^*\beta$.) The horizontal augmentation map $\si$ is  simply the  inclusion of the invariant forms \eqref{eq:invforms}, while $\sj$ is the pullback of $V$-valued functions 
on $G^p$ under the map $G/K\to \pt$.

Suppose now that $G$ has finitely many components and that $K$ is a \emph{maximal} compact subgroup of $G$. Recall that maximal compact subgroups are unique up to conjugation,   and that the homogeneous space $G/K$ is
diffeomorphic to the vector space $\g/\mf{k}$ (see, e.g., Borel \cite[Chapter VII]{bor:sem}).
For $G$ semisimple, there is a canonical such diffeomorphism, by the Cartan decomposition $G=K\exp(\mf{p})$. 

Under the diffeomorphism $G/K\cong \g/\mf{k}$, the scalar multiplication of $\g/\mf{k}$ translates into a  smooth deformation retraction 
\begin{equation}\label{eq:lambdamap} 
\lambda\colon [0,1]\times G/K\to G/K,\ \ 
(t,gK)\mapsto \lambda_t(gK),\end{equation}
with
$\lambda_{t_1t_2}=\lambda_{t_1}\circ \lambda_{t_2}$, interpolating between 
the identity map and the map $\iota_{eK}\circ \pr_{eK}$ where 
\[\pr_{eK}\colon G/K\to \pt,\ \ \iota_{eK}\colon \pt\to G/K\] 
are projection to and inclusion of the base point $eK\in G/K$.  It determines a {de Rham homotopy operator} on $V$-valued forms, 
\[ T\colon \Omega^{q}(G/K,V)\to \Omega^{q-1}(G/K,V),\] 
given by pullback under \eqref{eq:lambdamap}, followed by integration over $[0,1]$.  Thus $[\d_{Rh},T]=
\id-\pr_{eK}^*\circ\, \iota_{eK}^*$. The homotopy operator has the properties $T\circ T=0$, as well as 
\[ (T\beta)|_{eK}=0\]
for all $\beta\in \Omega^q(G/K,V)$. 
\begin{equation}\label{eq:kbeta}
(\k\beta)(g_1,\ldots,g_p)=(-1)^p T(\beta(g_1,\ldots,g_p)).
\end{equation}
 Thus $[\k,\d]=1-\sj\circ \sq$ where 
 \[ (\sq\beta)(g_1,\ldots,g_p)=\iota_{eK}^*\ \beta(g_1,\ldots,g_p).\]
The vertical homotopy operator determines a `van Est integration map'  
\[ R_{G/K}=\sq\circ (1+\delta\k)^{-1}\circ \si\colon 
\sC^\bullet(\g,K,V)
\to \sC^\bullet(G,V).\] 
\begin{proposition}\label{prop:rightinversegk}
The van Est integration map	$R_{G/K}$ is a right inverse to the van Est differentiation map $\VE_{G/K}$. 
\end{proposition}
\begin{proof}
By 	Lemma \ref{lem:backandforth}, it suffices to show that 
$\h\circ \k=0$ and $\sp\circ \k=0$. Given $\beta$ as in \eqref{eq:betamap}, let $\psi\colon 
G^p\times G\to V\otimes (\wedge^q\g^*)_{\mf{k}-\on{hor}}$ be the corresponding $K$-equivariant map.  
The formula \eqref{eq:hmap} shows that $\h\psi=0$ when $\psi|_{G^p\times K}=0$. 
Consequently, in the differential form picture, $\h\beta=0$ whenever
$\beta(g_1,\ldots,g_p)|_{eK}=0$ for all $(g_1,\ldots,g_p)\in G^p$. In particular, this applies when $\beta$ is in the range of $T$. 
 This shows $\h\circ \k=0$; the argument for $\sp\circ \k=0$ is similar.  
\end{proof}

We will now give a more explicit description of $R_{G/K}$. 
For $(g_1,\ldots,g_p)\in G^p$ and $(t_1,\ldots,t_p)\in [0,1]^p$ let 
\begin{equation}\label{eq:gammaformula0}
\gamma^{(p)}_{t_1,\ldots,t_p}(g_1,\ldots,g_p)=\big(\lambda_{t_1}\circ L(g_1)\cdots \circ \lambda_{t_p}\circ L(g_p)\big)(eK).
\end{equation}
For fixed $(g_1,\ldots,g_p)$, this defines a map $\gamma^{(p)}(g_1,\ldots,g_p)\colon [0,1]^p\to G/K$. \medskip

\begin{proposition}
Given $\alpha\in \sC^p(\g,V)_{K-\on{basic}}$, let $\alpha_{G/K}\in \Omega^p(G/K,V)^G$ 	be the corresponding $G$-equivariant form. Then 
\begin{equation}\label{eq:veintk}
R_{G/K}(\alpha)(g_1\ldots,g_p)=\int_{[0,1]^p}\gamma^{(p)}(g_1,\ldots,g_p)^*\alpha_{G/K}\end{equation} 
\end{proposition}
\begin{proof}
We will calculate $R_{G/K}(\alpha)=(-1)^p \sq\circ (\delta\k)^p\,  \si\ \alpha$ for $\alpha\in \sC^p(\g,V)$. We have that
\[ \si\alpha=\alpha_{G/K},\] 
viewed as an element of $\Omega^p(G/K,V)=\sD^{0,p}(G,V)_{K-\on{basic}}$. Next, $\delta\k\si \alpha\in \sD^{1,p-1}(G,V)_{K-\on{basic}}$ is given by 
\[ (\delta\k\si \alpha)(g_1)=T\ \alpha_{G/K}-L(g_1)^*T\ \alpha_{G/K}.\]
The next application of $\delta\k$ (or of $\sq$, if $p=1$) will annihilate the first term, due to 
$T\circ T=0$ (respectively, due to $\iota_{eK}^*\circ T=0$). Hence we only need to keep the second term, and we find that 
$(\delta\k)^2\si \alpha\in \sD^{2,p-2}(G,V)_{K-\on{basic}}$ is given by 
\[ ((\delta\k)^2\si\alpha)(g_1,g_2)=
T \circ L(g_2)^*\circ T\ \alpha_{G/K}
-T \circ L(g_1g_2)^*\circ T\ \alpha_{G/K}+L(g_2)^*\circ T\circ  L(g_1)^*\circ T \ \alpha_{G/K}.
\]
By the same reasoning as before, we need only keep the last term, since all terms starting with $T$ will be annihilated by the subsequent application of $\k$ (respectively $\sq$, if $p=2$). 
Proceeding in this manner, we arrive at the formula
 \[ R_{G/K}(\alpha)(g_1\ldots,g_p)=
\iota_{eK}^* \circ L(g_p)^*  \circ T\circ \cdots \circ L(g_1)^*\circ T\ \alpha_{G/K}.\]
(The $(-1)^p$ sign in the formula for $R_{G/K}$ is compensated by the alternating signs 
in the $L(g_i)^*\circ T$ contributions.) 
By definition, each $T$ involves pullback under the map $\lambda$, followed by integration over $[0,1]$. Denoting by $t_i$ the  variable for the $i$-th such integration, and by $\lambda^{(i)}$ the map corresponding to $\lambda$  for the variable $t_i$, 
we arrive at 
\[R_{G/K}(\alpha)(g_1\ldots,g_p)=  \int_{t_p \in [0,1]}\cdots \int_{t_1\in [0,1]}\iota_{eK}^* L(g_p)^* (\lambda^{(p)})^* \cdots 
L(g_1)^* (\lambda^{(1)})^* \alpha_{G/K}.
\] 
On the other hand, by definition, 
\[ \lambda^{(1)}\circ L(g_1)\circ \cdots \lambda^{(p)}\circ L(g_p)\circ \iota_{eK}=
\gamma^{(p)}(g_1,\ldots,g_p).\]
This gives \eqref{eq:veintk} (with the orientation of $[0,1]^p$ given by the volume element 
$\d t_p\wedge \cdots \wedge \d t_1$). 
\end{proof}

 In summary, we obtain the following cochain-level version of van Est's theorem:

\begin{theorem}[Van Est theorem for Lie groups] \label{th:vanest}
	Suppose $K$ is a maximal compact subgroup of the Lie group $G$. Then the van Est differentiation map \[ \VE_{G/K}\colon \sC^\bullet(G,V)\to 
	\sC^\bullet(\g,V)_{K-\on{basic}}\] defined by \eqref{eq:vek}
	is a homotopy equivalence. The van Est integration map 
	\[ R_{G/K}\colon \sC^\bullet(\g,V)_{K-\on{basic}}\to \sC^\bullet(G,V)
	\]
	 given by \eqref{eq:veintk} is a right inverse \emph{at the level of cochains}. 
\end{theorem}

\begin{remark}
Sometimes, it is convenient to work with the \emph{normalized subcomplex} $\wt{\sC}(G,V)$, consisting of functions $f\in C^\infty(G^p,V)$ with the property that $f(g_1,\ldots,g_p)=0$ whenever $g_i=e$ for some $i$.The inclusion of the normalized subcomplex is well-known to be a homotopy equivalence (see, e.g., \cite{mo:no}). Theorem \ref{th:vanest} holds for the normalized subcomplex, with the same proof. 
\end{remark}

	If $K$ is any compact Lie subgroup of $G$ (not necessarily maximal compact), one obtains a similar conclusion by replacing $\sC(G,V)$ 
	with the \emph{localized complex} 
	\[ \sC(G,V)_e=\sC(G,V)/\sim\]
	of \emph{germs of functions} 
	$G^p\to V$ at $(e,\ldots,e)\in G^p$, and using a germ at $eK$ of a diffeomorphism  $G/K\to \g/\mf{k}$ to define a germ of a retraction $\lambda_t$. One hence obtains a homotopy equivalence 
	\[ \VE_{G/K}\colon \sC^\bullet(G,V)_e\to 
	\sC^\bullet(\g,V)_{K-\on{basic}}\]
	with a homotopy inverse $R_{G/K}$ which is also a right inverse. In particular, this is true for $K=\{e\}$, cf.~ \cite{hou:int,swi:coh}.


\section{Van Est theory for Lie groupoids}
\label{sec:vanest}
We will next review the cochain complexes for Lie algebroids and Lie groupoids, and the van Est double complex connecting them. We then show how certain horizontal and vertical homotopy operators on the double complex define  van Est differentiation and integration maps, and finally show that the integration map is right inverse to the differentiation. Only at the end, we will derive the `explicit formulas' for the integration and differentiation. For basic information on Lie groupoids and Lie algebroids, we refer to  \cite{cra:lect,duf:po,moe:fol}. The van Est map for Lie groupoids was introduced by Weinstein-Xu in \cite{wei:ext} and further studied by Crainic \cite{cra:dif}; for further generalizations and applications see, \cite{aba:wei,cab:hom,cab:loc,lib:ve,meh:qgr,pfl:loc}.

\subsection{The simplicial manifold $B_pG$}
Let $G\rra M$ be a Lie groupoid, with source and target maps denoted $\sz,\tz\colon G\to M$.  Elements $g,h\in \G$ are \emph{composable} if $\sz(g)=\tz(h)$; in this case their groupoid product is denoted as $gh$. We denote by 
\[ B_pG=\{(g_1,\ldots,g_p)|\ \sz(g_i)=\tz(g_{i+1}),\ \ 0< i<p\}\]
the space of \emph{$p$-arrows}; by convention $B_0G=M$. Every $p$-arrow comes with $p+1$ base points $(m_0,\ldots,m_p)$, where $m_i=\sz(g_i)=\tz(g_{i+1})$. 
The collection of spaces $B_pG$ defines a simplicial manifold $B_\bullet G$ called the \emph{nerve} of the groupoid. The face map
$\partial_i\colon B_p\G\to 
B_{p-1}\G$
drops the $i$-th base point: 
\[  \partial_i(g_1,\ldots,g_p)=\begin{cases}
(g_2,\ldots,g_p), & i=0,\\
(g_1,\ldots,g_ig_{i+1},\ldots,g_p),&0<i<p,\\
(g_1,\ldots,g_{p-1}), &i=p,
\end{cases}
\]
while the $i$-th degeneracy map $\eps_i\colon B_pG\to B_{p+1}G$ repeats the $i$-th base point  by inserting a trivial arrow:
\[\eps_i(g_1,\ldots,g_p)=(g_1,\ldots,g_{i},m_i,g_{i+1},\ldots,g_p),\ \ i=0,\ldots,p.\]
%
Given a $G$-action on manifold $Q$, with anchor $\Phi\colon Q\to M$, 
one obtains a simplicial manifold 
\begin{equation}\label{eq:bpq}
B_pG\times_M Q
\end{equation}
where the fiber product is with respect to $\Phi$ and the map taking the $p$-arrow 
$(g_1,\ldots,g_p)\in B_pG$
to the base point $m_p$. 
%
The face and degeneracy maps maps are 
\begin{align*}
  \partial_i(g_1,\ldots,g_p;x),&=\begin{cases}
(g_2,\ldots,g_p;x), & i=0,\\
(g_1,\ldots,g_ig_{i+1},\ldots,g_p;x),&0<i<p,\\
(g_1,\ldots,g_{p-1};g_p x), &i=p,\end{cases}\\
\epsilon_i(g_1,\ldots,g_p;x)&= (g_1,\ldots,g_i,m_i,g_{i+1},\ldots,g_p;x),\ \ \ 0\le i\le p.
\end{align*}
(More conceptually, these formulas are explained through the identification $B_pG\times_M Q\cong B_p(G\ltimes Q)$, where $G\ltimes Q\rra Q$ is the action groupoid.)  
The manifolds $B_pG\times_M Q$ 
come equipped with 
$p+1$ commuting $G$-actions  (cf.~ \eqref{eq:agg}):
\begin{equation}\label{eq:actions} 
a\cdot (g_1,\ldots,g_p;x)=\begin{cases}
(ag_1,g_2,\ldots,g_p;x) & i=0,\\
(g_1,\ldots,g_i a^{-1},a g_{i+1},\ldots,g_p;x)&0<i<p,\\
(g_1,\ldots,g_{p-1},g_p a^{-1};a\cdot x) &i=p.
\end{cases}
\end{equation}
with  anchor map $(g_1,\ldots,g_p;x)\mapsto m_i$.

\subsection{The van Est double complex}
\subsubsection{Groupoid complex}
Given a representation of the groupoid $G\rra M$ on a vector bundle $V\to M$, taking $Q=V$ in \eqref{eq:bpq}, 
we obtain a simplicial vector bundle $ B_pG\times_M V\to B_pG$. 
The groupoid cochain complex $\big(\sC(G,V),\delta\big)$ has graded components 
\[ \sC^p(G,V)=\Gamma(B_pG\times_M V),\]
while the differential $\delta$ is given on $p$-cochains by $\delta=\sum_{i=0}^{p+1}(-1)^i \partial_i^*$. 
In the case of trivial coefficients $V=M\times \R$, we write $\sC(G)=\sC(G,M\times \R)$. 
\subsubsection{Lie algebroid complex}
Let $A\Ra M$ be the Lie algebroid of $G\rra M$. Thus, $A$ is the vector bundle whose sections are the left-invariant vector fields on $G$ (tangent to the $\tz$-fibers); for $\xi\in \Gamma(A)$ we denote by $\xi^L$ the corresponding left-invariant vector field. The anchor map $\a\colon A\to TM$ is characterized by its property 
$\xi^L\sim_\sz -\a(\xi)$. The $G$-representation on $V$ determines an $A$-representation on $V$. (Every $\xi\in \Gamma(A)$ is realized as the 
derivative of a 1-parameter family of bisections of $G$. 
The group of bisections acts linearly on the sections of $V$, and by differentiation one obtains the flat $A$-connection $\xi\mapsto \nabla_\xi$ defining the $A$-representation.) Let 
$\big(\sC(A,V),\d_{CE}\big)$
be the resulting \emph{Lie algebroid complex} (or \emph{Chevalley-Eilenberg complex}), 
with $p$-cochains 
\[ \sC^p(A,V)=\Gamma(V\otimes\wedge ^p A^* ),\] 
and with the Chevalley-Eilenberg differential $\d_{CE}$.  
%
For $\xi\in\Gamma(A)$, we denote by $\iota_\xi$ the operator on $\sC(A,V)$ given by contraction, and by $\L_\xi=[\d_{CE},\iota_\xi]$ the Lie derivative. On 
$\sC^0(A,V)=\Gamma(V)$, we have that $\L_\xi=\nabla_\xi$. In the case of the trivial representation on $V=M\times \R$, the connection is $\nabla_\xi=\L_{\a(\xi)}$ where $\a\colon A\to TM$ is the anchor of $A$; we will write $\sC(A)=\sC(A,M\times \R)$.\medskip

\subsubsection{Double complex}
The two complexes $\sX=\sC(A,V),\  \, \sY=\sC(G,V)$ are related by a van Est double complex, due to Crainic \cite{cra:dif}.  Taking $Q=G,\ \Phi=\tz$ in \eqref{eq:bpq}, with the $G$-action by left multiplication,  
define a simplicial fiber bundle 
\begin{equation}\label{eq:kappap}
 \kappa_p\colon E_p G=B_pG\times_M G\to B_pG.\end{equation}
For each $p$ this is a principal $G$-bundle, with 
anchor map 
\[ \pi_p(g_1,\ldots,g_p;g)=\sz(g),\] 
and principal action 
\[ a\cdot (g_1,\ldots,g_p;g)=(g_1,\ldots,g_p;ga^{-1}).\] 
(For background on principal bundles for Lie groupoids, see for example \cite{moe:fol}.) Each 
of these principal bundles is actually trivial, with a trivializing section 
\begin{equation}\label{eq:up}
 \nu_p\colon B_pG\to E_pG,\ \ (g_1,\ldots,g_p)\mapsto (g_1,\ldots,g_p;m)\end{equation}
where $m=\sz(g_p)$.

The face and degeneracy maps are principal bundle morphisms, making $\kappa\colon E G\to B G$ into a simplicial principal bundle; the `universal bundle' of the Lie groupoid $G$. (Note that \eqref{eq:up} are \emph{not} simplicial maps, and indeed $EG$ is non-trivial as a simplicial principal bundle.) The 
van Est double complex 
\[ (\sD^{\bullet,\bullet}(G,V),\delta,\d)\] 
is defined as follows. 
\begin{itemize}
\item The bigraded summands of the double complex
are 
\begin{equation}\label{eq:2altdef} 
\sD^{p,q}(G,V)=\Gamma\big(\pi_p^*(V\otimes \wedge^q A^*)\big),\end{equation}
generalizing the description \eqref{eq:dpq} in the case of Lie groups. 
\item 
$\delta$ is the simplicial differential on sections of the simplicial vector bundle 
\[ \pi_\bullet^*(\wedge^q A^*\otimes V)\to E_\bullet G.\] 
\item 
$\d=(-1)^p\d_{CE} $ on elements of bidegree $(p,q)$, with the Chevalley-Eilenberg differential on 
\begin{equation}\label{eq:dlacomplex}
\sD^{p,\bullet}(G,V)\cong \sC^\bullet(\pi_p^*A,\pi_p^*V)\end{equation}
In more detail, let $\F$ be the foliation of $E_pG$ given by the $\kappa_p$-fibers; thus $T_\F E_pG$ is the vertical bundle. The isomorphism 
$\pi_p^* A\cong T_\F E_pG$ defines a Lie algebroid structure on 
$\pi_p^*A$. On the other hand, the  isomorphism $\tz^*V\cong \sz^*V$ given by the $G$-representation extends to an isomorphism, for any $p$, 
\begin{equation}\label{eq:kappapi}
\kappa_p^*(B_pG\times_M V)\cong \pi_p^*V.
\end{equation} Since this bundle is trivial along the  leaves of $\F$, it comes 
with a natural representation of the Lie algebroid $\pi_p^*A=T_\F E_pG$, and the right hand side of \eqref{eq:dlacomplex} is its Lie algebroid complex.
\end{itemize}
Furthermore, the double complex comes with horizontal and vertical augmentation maps: 
\begin{itemize}
\item  $\si\colon \sC^\bullet(A,V)\to 
\sD^{0,\bullet}(G,V)$ is given in degree $q$ by the pullback $\pi_0^*$ (using \eqref{eq:dlacomplex} for $p=0$). 
%
\item 
$\sj\colon \sC^\bullet(\G,V)
\to \sD^{\bullet,0}(G,V)$ is given in degree $p$ by the pullback map $\kappa_p^*$ (using \eqref{eq:kappapi}).
%
\end{itemize}
The augmentation maps define cochain maps to the total complex
\[ \sC^\bullet(G,V)\stackrel{\sj}{\lra} \on{Tot}^\bullet(\sD(G,V))\stackrel{\si}{\longleftarrow}\sC^\bullet(A,V) .\]

\begin{remark}
Sometimes, it is better to work with the \emph{normalized subcomplex}, defined by the requirement that all pull-backs under the degeneracy maps $\epsilon_i$ are equal to zero. 
We will indicate the normalized 
subcomplexes (and the spaces of functions and sections defining them) by a tilde; for example 
\[ \wt{\sC}^\bullet(G,V)=\wt{\Gamma}(B_\bullet G\times_M V),\ \ 
\wt{\sD}^{\bullet,\bullet}(G,V)=\wt{\Gamma}(\pi_p^*(V\otimes \wedge^q A^*))
.\] 
By a general result for simplicial manifolds (see e.g. \cite{mo:no}),   
the inclusion $\wt{\sC}(G,V)\hra \sC(G,V)$ is  a homotopy equivalence.

As another variation, we will consider \emph{localized} versions of these complexes, with respect to the submanifold $M\subset B_pG$ of constant $p$-arrows. These cochain complexes 
\[ \sC^\bullet(G,V)_M={\Gamma}(B_\bullet G\times_M V)_M,\] 
are given by \emph{germs of sections} along the submanifold $M\subset B_\bullet G$. 
There is also a localized version $\sD^{\bullet,\bullet}(G,V)_M$ of the double complex
(and its normalized subcomplex), working with germs of sections along $M\subset E_\bullet G$. Note that the localized version (as well as its normalized subcomplex) also makes sense for \emph{local Lie groupoids}.  
\end{remark}

\subsection{Differentiation}
\subsubsection{The horizontal homotopy} 
The simplicial universal bundle $E G$ comes with a simplicial retraction onto its submanifold $M$ (see \cite{se:cl} and \cite[Appendix A.2]{lib:ve}). This is reflected in the existence of a homotopy operator on the double complex 
$\sD(G,V)$. Consider the maps
\[ h_p\colon E_pG\to E_{p+1}G,\ (g_1,\ldots,g_p;g)\mapsto (g_1,\ldots,g_p,g;m)\]
where $m=\sz(g)$. Since  $\pi_{p+1}\circ h_{p}=\pi_{p}$, 
these lift to fiberwise isomorphisms of vector bundles $\pi_{p}^*(\wedge^q A^*\otimes V)\to \pi_{p+1}^*(\wedge^q A^*\otimes V)$, defining a  pullback map on sections. 
\begin{lemma}
The map 
\[ \h\colon \sD^{p,q}(G,V)\to \sD^{p-1,q}(G,V),\ \ \psi\mapsto (-1)^p h_{p-1}^*\psi\]
satisfies 
\[ [\h,\delta]=1-\si\circ \sp,\]
where $\sp\colon \sD^{0,\bullet}(G,V)\to \sC^\bullet(A,V)$ is the 
left inverse to $\si=\pi_0^*$
given by pullback under the inclusion  $u\colon M\hra E_0\G=G$:
\[ \sp=u^*\colon \Gamma(\pi_0^*(\wedge^q A^*\otimes V))\to \Gamma(\wedge^q A^*\otimes V).\] 
\end{lemma}
\begin{proof}
Let $\psi\in \sD^{p,q}(G,V)$. If $p=0$ we have that $\h\psi=0$, while $(\delta\psi)(g_1;g)=
\psi(g)-\psi(g_1g)$ and therefore $(h\delta\psi)(g)=-(\delta\psi)(g;m)=\psi(g)-\psi(m)$. 
For $p<0$ we have that 
$(\h\psi)(g_1,\ldots,g_{p-1};g)=
(-1)^p\psi(g_1,\ldots,g_{p-1},g;m)$ and therefore
\begin{align*}
(\delta\h\psi)(g_1,\ldots,g_p;g)&=
(-1)^{p}\Big( 
\psi(g_2,\ldots,g_p;m)-\psi(g_1g_2,\ldots,g_p,g;m)\pm \cdots \\
&\ \ \ +(-1)^{p}\psi(g_1,\ldots,g_{p-1},g;m)\Big)
\Big)
\end{align*}	
	Similarly, 
\begin{align*}
(\h\delta\psi)(g_1,\ldots,g_p;g)&=
(-1)^{p+1}\Big( 
\psi(g_2,\ldots,g_p;m)-\psi(g_1g_2,\ldots,g_p,g;m)\pm \cdots \\
&\ \ \ +(-1)^{p+1}\psi(g_1,\ldots,g_p;g)\Big)
\Big).
\end{align*}	
Adding the two expressions, all terms except for $\psi(g_1,\ldots,g_p;m)$ cancel. 
\end{proof}

Note that since $\kappa_{p+1}\circ h_p\neq \kappa_p$, in general, the maps $h_p$ need not preserve the 
foliation $\F$, and hence the homotopy operator $\h$ and the projection $\sp$ do not usually 
commute with the differential $\d$. 
\begin{remark}
	The homotopy operator $\h$ and the projection $\sp$ restrict to the normalized subcomplex $\wt{D}^{\bullet,\bullet}(G,V)$. On this subcomplex , they have the additional properties
	\begin{equation}\label{eq:sideconditions} \h\circ \h=0,\ \ \ \ \sp\circ \mathsf{h}=0;\end{equation}
	this follows because $h_p$ coincides on the range of $h_{p-1}$ (or of $u\colon M\hra G_0$, in case $p=0$) with the degeneracy map $\epsilon_p$.
\end{remark}

\subsubsection{Van Est map}
The Perturbation Lemma \ref{lem:perturbed} gives a new projection $\sp'=\sp\circ (1+\d \h)^{-1}$,  which is a cochain map for the total differential $\d+\delta$, with $\sp'\circ \si=\id$. Thus, $\si$ is a homotopy equivalence, with $\sp'$ a homotopy inverse.  We obtain a cochain map 
\begin{equation}\label{eq:veprime}
\on{VE}_G=	\sp \circ (1+\d \h)^{-1}\circ \sj\colon \sC^\bullet(\G,V)\to 
\sC^\bullet(A,V).
\end{equation}
For a more explicit description of this map, recall the commuting $G$-actions \eqref{eq:actions} on $B_pG\times_M Q$, for a $G$-manifold $Q$. 
These actions have generating vector fields $\xi^{(i)},\ \xi\in \Gamma(A),\ i=1,\ldots,p$. In the case  
of $Q=V$, the $\xi^{(i)}$ are linear with respect to the vector bundle structure on $B_pG\times_M V\to B_pG$. They hence define covariant derivatives 
\[ \nabla_\xi^{(i)}\colon \Gamma(B_pG\times_M V)\to \Gamma(B_pG\times_M V).\]

\begin{theorem}\cite{lib:ve}\label{th:vedifferentiation}
The map $\VE_G$ is given by the formula, 
\begin{equation}\label{eq:weinsteinxu}
\VE_G(\sigma)(\xi_1,\ldots,\xi_p)=
\sum_{s\in \mf{S}_p} \on{sign}(s)\ 
\nabla^{(1)}_{\xi_{s(1)}}\cdots \nabla^{(p)}_{\xi_{s(p)}} (\sigma)\big|_M\end{equation}
for $\sigma\in\sC^p(G,V)=\Gamma(B_pG\times_M V)
$ and $\xi_1,\ldots,\xi_p\in\Gamma(A)$. Here the sum is over the permutation group $\mf{S}_p$, and $M$ is regarded as a submanifold of $B_pG$
consisting of constant $p$-arrows.
\end{theorem}

\begin{remark}
Equation \eqref{eq:weinsteinxu} is Weinstein-Xu's formula \cite{wei:ext} for the van Est map $\VE_G$. 
To be precise, \cite{wei:ext} only treated the case of trivial coefficients, and exclusively worked with the normalized subcomplex $\wt{\sC}(G)$. They proved by direct  computation that this expression defines a cochain map, and furthermore that it intertwines the cup product on 
groupoid cochains with the wedge product on Lie algebroid cochains. 
The latter fact only holds true on the normalized subcomplex. In \cite{lib:ve}, it was explained by additional properties of the homotopy operator on the normalized  sub-double complex, such as \eqref{eq:sideconditions}.  
\end{remark}

We include a proof of Theorem \ref{th:vedifferentiation} in the Appendix.  (It is a slightly simplified version of the argument in \cite{lib:ve}.)

\subsection{Integration}\label{subsec:integrationG}
We next discuss the  integration from Lie algebroid cochains to Lie groupoid cochains. We will work with the localized complex $\sC^\bullet(G,V)_M$ defined in terms of germs of sections along $M\subset B_pG$; here $G$ could also be only a local Lie groupoid. For  convenience, we will typically omit explicit emphasis of `germs' and `local'. If $G$ is a Lie groupoid which happens to be \emph{globally} contractible to $M$ along its $\tz$-fibers, one may work with the complex $\sC(G,V)$. 
 
\subsubsection{Differential form picture of double complex} \label{subsec:diffform}
 Just as in the case that $G$ is a Lie group,  the discussion of integration is more convenient using an interpretation in term of differential forms. 
 Let  $(\Omega_\F(G,\tz^*V),\d_{Rh})$ be the de Rham complex of foliated (leafwise) $\tz^*V$-valued forms 
 on $G$. Restriction  to $M\subset G$ takes such a form to a section of $V\otimes \wedge A^*$, and induced an isomorphism of differential complexes, 
  \begin{equation}\label{eq:dcedrh}(\Omega^\bullet_\F(G,\tz^*V)^L,\d_{Rh})\cong
   (\sC^\bullet(A,V),\d_{CE}) .\end{equation}
Observe furthermore that 
 \begin{align*}\pi_p^*A&=T_\F E_pG=B_pG\times_M T_\F G,\\
  \pi_p^*V&=\kappa_p^*(B_pG\times_M V)=B_pG\times_M \tz^*V.\end{align*}
 We may hence regard the elements of $\sD^{p,q}(G,V)$ 
  as  maps 
 \[ \beta\colon B_pG\to \Omega_\F^q(G,\tz^* V),\ (g_1,\ldots,g_p)\mapsto \beta(g_1,\ldots,g_p)\]
 such that $\beta(g_1,\ldots,g_p)\in \Omega^q(\tz^{-1}(m))\otimes V_m$ for $m=\sz(g_p)$, and smoothly depending on $(g_1,\ldots,g_p)$.  Similarly, 
 $\sD^{p,q}(G,V)_M$ is interpreted as germs along $M\subset B_pG$ of such maps. 
 
 In this picture, the vertical differential is $\d=(-1)^p \d_{Rh}$ 
 (cf. \eqref{eq:dbeta}), while the horizontal differential is described similar to \eqref{eq:deltabeta}. The augmentation map $\si$ is the inclusion of \eqref{eq:dcedrh},  
 \[ \si\colon \Omega^q_\F(G,\tz^*V)^L\hra \sD^{0,q}(G,V)=\Omega^q_\F(G,\tz^*V)\]
 	while 
 	$\sj$ is the inclusion of $\Gamma(B_pG\times_M V)$ into the space of 
 	maps $\beta\colon B_pG\to \Omega^0_\F(G,\tz^*V)$ such that $\beta(g_1,\ldots,g_p)\in C^\infty(\tz^{-1}(m))\otimes V_m$ is constant on $\tz^{-1}(m)$, for any given $(g_1,\ldots,g_p)$ with $m=\sz(g_p)$.

 \subsubsection{The integration map $R_G$}
 A \emph{tubular structure} for a (local) Lie groupoid $ G\rra M$ is a tubular neighborhood embedding
 \[ A\cong \ker(T\tz)|_M\to  G\] 
 taking the fibers of 
 $A\to M$ to the $\tz$-fibers, and with differential along $M$ the identity map of $A$. 
 The tubular structure transports the scalar multiplication in $A$ to a retraction along $\tz$-fibers 
 \begin{equation}\label{eq:lambdaretraction}
  \lambda\colon [0,1]\times  G\to  G,\ (t,g)\mapsto \lambda_t(g),
  \end{equation}
 or more precisely the germ along $[0,1]\times M$ of such a map.  Here 
  $\lambda_0=u\circ \tz$, where $u\colon M\to G$ is 
  the inclusion of units. 
  The retraction determines a homotopy operator 
\[ T\colon \Omega^q_\F(G,\tz^*V)_M\to \Omega^{q-1}_\F(G,\tz^*V)_M\]
given by pullback under $\lambda$ followed by integration over $[0,1]$. 
This has the properties $T\circ T=0$ and 
\[ T\beta|_M=0.\]
Similar to 
\eqref{eq:kbeta}, it defines a vertical homotopy operator $\k$ on the double complex, where
\[ (\k\beta)(g_1,\ldots,g_p)=(-1)^p T(\beta(g_1,\ldots,g_p))\]
That is, $[\k,\d]=1-\sj\circ \sq$ where, for $\beta$ of bidegree $(p,0)$, 
\[ (\sq\beta)(g_1,\ldots,g_p)=u^* \beta(g_1,\ldots,g_p)\] 
(the restriction of $\beta(g_1,\ldots,g_p)\in \Gamma(\tz^* V)$ to the units). The properties of $T$ show that 
\[ \k\circ \k=0,\ \ \sq\circ \k=0.\]
\smallskip

By the Perturbation Lemma \ref{lem:perturbed}, we obtain a cochain map 
\begin{equation} R_G=\sq\circ (1+\delta\k)^{-1}\circ \si\colon \sC^\bullet(A,V)\to \sC^\bullet(G,V)_M. \end{equation}
Note again that on elements of degree $p$, the map $R_{G}$ is given by a zig-zag $(-1)^p \sq\circ (\delta \k)^p\circ \si$. 
\begin{remark}\label{rem:normalized}
The vertical homotopy  $\k$ restricts to the normalized sub-double complex $\wt\sD(G,V)_M$. Hence, $R_G$ takes values in the normalized subcomplex $\wt{\sC}(G,V)_M$.
\end{remark}

\begin{proposition}\label{prop:rightinverse}
The integration map $R_G$ is right inverse to the van Est differentiation $\VE_G$:
\[ \VE_G\circ R_G=\id_{\sC(A,V)}.\]
\end{proposition}
\begin{proof}
Let $\psi\in \sD^{p,q}(G,V)$ with the corresponding map $\beta\colon B_pG\to \Omega^q_\F(G,\tz^* V)$. By the properties of $T$,
\[ (\k\beta)(g_1,\ldots,g_p)|_M=0.\]
This means that the section $\k\psi
\in \Gamma(\pi_p^*(V\otimes \wedge^{q-1} A^*))$ corresponding to $\k\beta$ vanishes along 
$B_pG\times_M M\subset E_pG=B_pG\times_M G$. But then $\h\k\psi=(-1)^p h_{p-1}^*\k\psi=0$. 
This shows $\h\circ \k=0$; similarly we obtain $\sp\circ \k=0$. Now use
Lemma \ref{lem:backandforth}.
\end{proof}

 \subsubsection{A formula for $R_G$}
We will now show that the van Est integration map $R_G$ coincides with the map defined in 
\cite{cab:loc}. Define a a map $[0,1]^p\times B_p G\to G$ (more precisely, a germ along 
$[0,1]^p\times M$ of such a map) 
by the formula:
 \begin{equation}\label{eq:gammaformula}
  \gamma^{(p)}_{t_1,\ldots,t_p}(g_1,\ldots,g_p)=\lambda_{t_1}\Big(g_1\,\lambda_{t_2} \big(g_2\cdots 
 \lambda_{t_p}(g_p)\cdots\big)\Big).
 \end{equation}
 For fixed $(g_1,\ldots,g_p)$ (close to $M$), this is a smooth map from the unit cube 
 $[0,1]^p$ into the $\tz$-fiber of  $m_0=\tz(g_1)$. 
 \begin{theorem}\label{th:thesame}
 The van Est integration map $R_G=\sq\circ (1+\delta\k)^{-1}\circ \si$ is given on 
 degree $p$ elements $\alpha\in \sC^p(A,V)$ by the formula 
  \begin{equation}\label{eq:cabdef} R_{G}(\alpha)(g_1,\ldots,g_p)=\int_{[0,1]^p}(\gamma^{(p)}(g_1,\ldots,g_p))^*(\alpha^L|_{\tz^{-1}(m_0)}),
  \end{equation}
 \end{theorem}
 Here $\alpha^L\in \Omega^p_\F(G,\tz^*V)_M$ is the left-invariant foliated form defined by $\alpha\in \sC^p(A,V)$. We think of  $\alpha^L$ as a family of $V_m$-valued forms on the fibers $\tz^{-1}(m)$; for fixed $(g_1,\ldots,g_p)$ the map $\gamma^{(p)}(g_1,\ldots,g_p)$ takes values in 
 one such fiber, hence the pull-back is an ordinary form on $[0,1]^p$. 
 
 \begin{remark}
  In \cite{cab:loc}, it was shown by direct calculation that the right hand side of \eqref{eq:cabdef} is a cochain map, which is a right inverse to $\VE_G$ at the level of cochains.
 \end{remark}

\subsubsection{Proof of Theorem \ref{th:thesame}}
The proof will require some preliminary results. Observe first the following alternative description
of the maps $\gamma^{(p)}$. Denote by $\lambda^{(p)}_{t_p}$ the map $E_{p-1}G\to E_{p-1}G$
given by $(g_1,\ldots,g_{p-1},g)\mapsto (g_1,\ldots,g_{p-1},\lambda_{t_p}(g))$. 

\begin{lemma}\label{lem:gammalemma}
	The map \eqref{eq:gammaformula} is a composition
	\[ \gamma^{(p)}_{t_1,\ldots,t_p}=
	\lambda^{(1)}_{t_1}\circ \partial_1\circ \lambda^{(2)}_{t_2}\circ \partial_2\circ \cdots \circ \lambda^{(p)}_{t_p}\circ \partial_p\circ \nu_p\colon B_p\G\to G.\]
\end{lemma}
\begin{proof} 
By direct calculation, 
\begin{align*}
 (g_1,\ldots,g_p) & \stackrel{\nu_p}{\longmapsto}   \big(g_1,\ldots,g_p;\, \sz(g_p)\big)\\ 
 & \stackrel{\partial_p }{\longmapsto}  (g_1,\ldots,g_{p-1};g_p)\\
 & \stackrel{\lambda_{t_p}^{(p)} }{\longmapsto} \big(g_1,\ldots,g_{p-1};\lambda_{t_p}(g_p)\big)\\
 & \stackrel{\partial_{p-1} }{\longmapsto} \big(g_1,\ldots,g_{p-2};g_{p-1}\lambda_{t_p}(g_p)\big)\\
& \stackrel{\lambda_{t_{p-1}}^{(p-1)} }{\longmapsto}
\Big(g_1,\ldots,g_{p-2};\lambda_{t_{p-1}}\big(g_{p-1}\lambda_{t_p}(g_p)\big)\Big)\\
 &\cdots
 \end{align*}
eventually arriving at   \eqref{eq:gammaformula}. 
\end{proof}
We will also need:
\begin{lemma}\label{lem:usuallynot}
	 For $0\le i\le p$ (but usually not for $i=p+1$), we have that 
\begin{align*}
 \partial_i^*\circ \k=\k\circ \partial_i^* &\colon \sD^{p,q}(G,V)_M\to \sD^{p+1,q-1}(G,V)_M,\\
 \partial_i^*\circ \sq=\sq\circ \partial_i^* &\colon \sD^{p,0}(G,V)_M\to \sC^{p+1}(G,V)_M.
\end{align*}
\end{lemma}
\begin{proof}
The identities follow since  $\lambda_t^{(p)}\circ \partial_i=\partial_i\circ \lambda_t^{(p+1)}\colon E_pG\to E_{p-1}G$ 
for $0\le i < p$  (but usually not for $i=p$), and 
$\nu_{p-1}\circ \partial_i=\partial_i\circ \nu_{p}\colon B_pG\to E_{p-1}G$ for $0\le i<p$  (but usually not for $i=p$).
\end{proof}

\begin{proof}[Proof of Theorem \ref{th:thesame}]
Let $\alpha\in \sC^p(A,V)=\Gamma(V\otimes \wedge^p A^*)$. Then
\[ R_{G}(\alpha)=(-1)^p\, \sq\circ (\delta\circ \k)^p\ \si(\alpha).\]
In this expression, the leftmost $\delta$ is the map $\delta=\sum_{i=0}^p (-1)^i \partial_i^*\colon 
\sD^{p-1,0}(G,V)_M\to \sD^{p,0}(G,V)_M$. 
Using $\sq\circ \k=0$  and Lemma \ref{lem:usuallynot}, we have that 
 $\sq\circ \partial_i^*\circ \k=
\partial_i^*\circ \sq\circ \k=0$ for $i<p$. Hence the composition $\sq\circ \delta$ may be replaced with 
$(-1)^p \sq\circ \partial_p^*=(-1)^p \nu_p^*\circ\partial_p^*$, leading to
\[ R_{G}(\alpha)=\nu_p^* \circ \partial_p^*\circ \k \circ (\delta\circ \k)^{p-1}\ \si(\alpha).\]
If $p>1$,  consider the leftmost product $\k \circ \delta \circ \k\colon 
\sD^{p-1,2}(G,V)_M
\to \sD^{p-1,0}(G,V)_M$.  Using $\k\circ \k=0$ and Lemma \ref{lem:usuallynot} again,  
$\k\circ \partial_i^*\circ \k=\partial_i^*\circ \k\circ \k=0$ for $i<p-1$; hence we may replace 
this expression with 
$(-1)^{p-1}\k\circ \partial_{p-1}^*\circ \k$. Continuing in this way, we arrive at 
\[
R_{G}(\alpha)=(-1)^{p(p-1)/2}
\big(\nu_p^*\circ \partial_p^*\circ \k\circ \partial_{p-1}^*\circ \k\circ \cdots 
\circ \partial_1^*\circ \k\big)\,\si(\alpha).
\]
But $\k\colon \sD^{p,q}(G,V)_M\to \sD^{p,q-1}(G,V)_M$ 
is given by $(-1)^p$ times pull-back under the map $\lambda^{(p)}\colon [0,1]\times E_p\G\to E_p\G$, 
followed by integration over $[0,1]$. The signs for $\k$'s cancel the 
$(-1)^{p(p-1)/2}$, and the resulting expression reads as
\begin{align*}
 R_{G}(\alpha)&=\int_{[0,1]^p}
\big(\nu_p^*\circ \partial_p^*\circ (\lambda^{(p)})^*\circ \partial_{p-1}^*\circ (\lambda^{(p-1)})^*\circ \cdots 
\circ \partial_1^*\circ (\lambda^{(1)})^*\big)\,\si(\alpha)\\
&=\int_{[0,1]^p} (\gamma^{(p)})^*\alpha^L,\end{align*} 
where we used Lemma \ref{lem:gammalemma} and $\si(\alpha)=\alpha^L$.
\end{proof}


\subsection{Example: the pair groupoid}
Let $\on{Pair}(M)\rra M$ be the pair groupoid of the manifold $M$, with associated Lie algebroid the tangent bundle $TM$.  Here $B_pG\cong M^{p+1}$, 
and $\sC^\bullet(\on{Pair}(M))_M=C^\infty(M^{\bullet+1})_M$ is the Alexander-Spanier complex, with differential 
\[ (\delta f)(m_0,\ldots,m_{p+1})=\sum_{i=0}^{p+1}(-1)^i 
f(m_0,\ldots,m_{i-1},m_{i+1},\ldots,m_{p+1}),\]
while $\sC^\bullet(TM)=\Omega^\bullet(M)$ is the usual de Rham complex.  
The Van Est differentiation is given on functions of the form $f=f_0\otimes \cdots \otimes f_p$ with $f_i\in C^\infty(M)$ by 
\[ \on{VE}(f_0\otimes \cdots \otimes f_p)=f_0\,\d f_1\wedge \cdots \wedge \d f_p.\]
For the integration, choose an affine connection on $M$. For
$m_0,m_1\in M$ sufficiently close, let 
 $\rho(m_0,m_1)\colon [0,1]\to M$ be the geodesic starting at $m_0$ and ending at 
 $m_1$.  Generalize to maps $\rho^{(p)}(m_0,\ldots,m_p)\colon [0,1]^p\to M$, given by 
 $\rho$ for $p=1$ and inductively by 
\[ \rho^{(p)}_{ t_1,\ldots,t_p }(m_0,\ldots ,m_p)=
\rho_{t_1}(m_0,\rho^{(p-1)}_{t_2,\ldots,t_p }(m_1,\ldots m_{p})).\]
The van Est integration $R_G$ takes $\alpha\in \Omega^p(M)$ to the (germ of a) function 
$(m_0,\ldots,m_p)\mapsto \int_{[0,1]^p} (\rho^{(p)}(m_0,\ldots,m_p))^*\alpha$.

\section{Van Est maps for Lie groupoid actions on manifolds}
\label{sec:actions}
In his paper \cite{cra:dif}, Crainic proved a general van Est theorem for (proper) groupoid actions on manifolds $Q$. We explain how to generalize the differentiation and integration maps  for cochains to this context. The construction of a horizontal homotopy operator for the double complex will require the additional data of a \emph{Haar distribution}.  

\subsection{Haar distributions}
By a (left-invariant) Haar distribution on a Lie groupoid $G\rra M$, we mean a family $\mu=\{\mu_m\}$ of distributions on the $\tz$-fibers 
\[ \mu_m\in \D'(\tz^{-1}(m)),\ \ \ m\in M,\]
such that: 
\begin{itemize}
\item The family depends smoothly on $m$, in the sense that for any compactly supported function $f\in C^\infty(G)$ the integral over $\tz$-fibers defines a smooth function 
\begin{equation}\label{eq:fiberintegral}
m\mapsto  \int_{a\in \tz^{-1}(m)} f(a)\,\mu_m(a).\end{equation}
\item The family is left-invariant, in the sense that 
$(l_g)_*\mu_m=\mu_{g\cdot m}$ for all $g\in G$ with $\sz(g)=m$. 
\end{itemize}
The Haar distribution is called \emph{properly supported} if $\tz$ restricts to a proper map 
$\on{supp}(\mu)\to M$; in particular, this means that the individual distributions $\mu_m$ are compactly supported.  
It is called \emph{normalized} if furthermore $\int_{\tz^{-1}(m)} \mu_m=1$ for all $m$, and \emph{non-negative} if the integral \eqref{eq:fiberintegral} is non-negative for all $f\ge 0$.   A Haar distribution is called a \emph{Haar density} if it is smooth; by left-invariance, these are equivalent to smooth sections of the density bundle of $A=\on{Lie}(G)$. It is known \cite{cra:dif,cra:orb,tu:con} that if the Lie groupoid is \emph{proper}, in the sense that $(\tz,\sz)\colon G\to M\times M$ is a proper map, then $G$ admits a properly supported, non-negative, normalized Haar density. 

As shown by Crainic \cite[Proof of Proposition 1]{cra:dif}, a properly supported normalized Haar distribution $\mu$ for a proper groupoid 
$G\rra M$ defines a homotopy operator for the groupoid cochain complex  $\sC(G,V)$, for any $G$-representation $V$:
\begin{equation}\label{eq:muhomotopy}
 (\h\sigma)(g_1,\ldots,g_{p-1})=(-1)^p \int_{a\in \tz^{-1}(m)} a\cdot \sigma(g_1,\ldots,g_{p-1},a)\,\mu_m(a)\end{equation}
where $m=\sz(g_{p-1})$. (Actually, \cite{cra:dif} only considers Haar densities, but the calculation for distributions is exactly the same.) Thus $[\h,\delta]=1-\si\circ \sp$, where $\si$ is the inclusion of invariant sections $\Gamma(V)^G$, while $\sp$ takes a section $\tau\in \Gamma(V)$ to the invariant section obtained by averaging:
\begin{equation}\label{eq:muprojection}
 (\sp\tau)_m=\int_{a\in \tz^{-1}(m)}  a\cdot \tau(\sz(a))\ \mu_m(a).
\end{equation}

\begin{examples}
\begin{enumerate}
\item For a Lie group $G$, every Haar distribution is automatically smooth, and is obtained by left translation of an element of 
the density space of $T_eG=\g$. The groupoid 
$G\rra \pt$ is proper (and its Haar measure is properly supported) 
if and only if $G$ is compact.
\item The pair groupoid $\on{Pair}(M)=M\times M\rra M$ is proper. Its $\tz$-fibers are $\tz^{-1}(m)=\{m\}\times M \cong M$; under this identification, the left-action of elements $a=(m',m)$ on $\on{Pair}(M)$ corresponds to the \emph{trivial} diffeomorphism of $M$. Hence, any fixed $\nu\in \D'(M)$ defines a Haar distribution with $\mu_m=\nu$ independent of $m$; it is proper if $\nu$ has compact support, and normalized if $\int \nu=1$. 
In particular, we may take $\nu$ to be the  delta-distribution at any given base point 
$z\in M$. The resulting homotopy for the 
complex $\sC^\bullet(\on{Pair}(M))=C^\infty(M^{\bullet+1})$ is the standard one:
\[ (\h f)(m_0,\ldots,m_{p-1})=(-1)^p 
f(m_0,\ldots,m_{p-1},z).\]
\item Let $\ca{U}=\{U_i\}$ be an open cover of a manifold $M$,
and put $X=\sqcup_i U_i$. The associated \v{C}ech groupoid 
$X\times_M X\rra X$, with the groupoid structure induced from the pair groupoid 
$\on{Pair}(X)$, is proper. A locally finite partition of unity $\{\chi_i\}$ subordinate to the cover $\ca{U}$ defines a normalized properly supported Haar distribution: the $\tz$-fiber of $m\in 
U_i$ is a disjoint union of copies of $\{m\}$ (one for each $U_j$ containing $m$), and the Haar density on this discrete set is given by the sequence  $\{\chi_j(m)\}$. The invariant elements of $\sC^0(X\times_M X)=C^\infty(X)$ 
are pullback of functions on $M$, and the homotopy operator on the \v{C}ech complex 
$\sC(X\times_M X)$ is the standard one \cite[Chapter 2.8]{bo:di}, cf.~ Example \ref{ex:chechderham}. 
\end{enumerate}
\end{examples} 

Suppose that $G\rra M$ is a Lie groupoid, and that $Q$ is a $G$-manifold, with anchor map $\Phi\colon Q\to M$. The action is called \emph{proper} if the action groupoid $G\ltimes Q\rra Q$ is proper. 
The $\tz$-fiber of $x\in Q$ in the action groupoid is canonically identified with the $\tz$-fiber of 
$\Phi(x)\in M$ in the groupoid $G$; hence a Haar distribution $\mu$ for the action groupoid  amounts to a family of distributions 
\begin{equation} \label{eq:actonq}\mu_x\in \ca{D}'(\tz^{-1}(\Phi(x))),\ \ x\in Q\end{equation}
with the following invariance property: 
\[ \mu_{a.x}=(l_{a})_*\mu_x,\ \ \ x\in Q,\ a\in G.\] 
If the action is proper, there exists a properly supported normalized Haar distribution. 

\begin{examples}\label{ex:actionhaar}
\begin{enumerate}
\item\label{it:1} For the action of $G$ on $Q=G$ by left translation, we may take \eqref{eq:actonq} to be 
the collection of delta-distributions
\[ \mu_g=\delta_g \in \ca{D}'(\tz^{-1}(\tz(g)).\]
\item\label{it:2} Let $G$ be a Lie group, and consider the homogeneous space $Q=G/K$ where 
$K$ a compact subgroup. Let $\delta_K\in \D'(G)$ be the push-forward of the normalized Haar density  on $K$. Then the family of distributions $\mu_{gK}=(l_g)_*\delta_K$ defines a normalized, properly supported Haar distribution.
\end{enumerate}
\end{examples}

\subsection{The van Est double complex}\label{subsec:double2}
For the rest of this section, suppose that $G\rra M$ is a Lie groupoid acting on a manifold $Q$, with moment map 
$\Phi\colon Q\to M$ a \emph{submersion}. The fibers of $\Phi$ define a $G$-invariant foliation $\F$ of $Q$, and a corresponding $G$-equivariant Lie algebroid $T_\F Q=\ker(T\Phi)\subset TQ$.
For any vector bundle $V\to M$, we obtain a `fiberwise trivial' Lie algebroid  representation of $T_\F Q$ 
on the vector bundle $\Phi^*V=Q\times_M V\to Q$; 
given a $G$-representation on $V\to M$, this representation is compatible with the $G$-action on $\Phi^*V$. It defines a foliated de Rham complex 
\begin{equation}\label{eq:leafwise}
 \Omega^\bullet_\F(Q,\Phi^*V)=\Gamma(  \Phi^*V    \otimes \wedge^\bullet T_\F^* Q).\end{equation} 
 The invariant subcomplex 
%
$\Omega^\bullet_\F(Q,\Phi^*V)^{G}$
consists of sections $\phi\in \Gamma(\Phi^*V\otimes \wedge T_\F^* Q)$ with the equivariance property $\phi(g\cdot q)=g\cdot \phi(q)$. 
(If $Q=G$ with the $G$-action by left translation, this is the complex
of left-invariant $\tz^*V$-valued forms on $G$, and is identified with $\sC(A,V)$. If $G$ is a Lie group and $Q=G/K$, this is the de Rham complex $\Omega(G/K,V)^G\cong 
\sC(\g,V)_{K-\on{basic}}$.) 
Let 
\[ \Phi_p\colon B_pG\times_M Q\to B_pG\]
be the natural projection (given by $\Phi$ if $p=0$).
The foliation of $Q$ extends to foliations $\F$ of $B_pG\times_M Q$ given by the fibers of $\Phi_p$.  

We obtain a simplicial Lie algebroid $T_\F(B_pG\times_M Q)=B_pG\times_M T_\F Q$, together 
 with a (fiberwise trivial) representation on the vector bundle $\Phi_p^*(B_pG\times_M V)\to B_pG\times_M Q$. Following Crainic \cite{cra:dif}, define a double complex 
\begin{equation}\label{eq:dq1} \sD^{p,q}(Q,V)=\sC^q(T_\F (B_pG\times_M Q),\Phi_p^*(B_pG\times _M V)
\end{equation}
with the usual simplicial differential $\delta$, and with $\d=(-1)^p\d_{CE}$. 
Its elements may be regarded as maps 
$\beta\colon B_pG\to \Omega^q_\F(Q,\Phi^*V)$, 
with 
\begin{equation}\label{eq:dqtrivializationB}
 \beta(g_1,\ldots,g_p)\in \Omega^q(\Phi^{-1}(m),V_m),\ \ \ \ m=\sz(g_p).\end{equation}
 with $\d=(-1)^p\d_{Rh}$ and $\delta$ as in \eqref{eq:deltabeta}. This double complex comes with a horizontal 
augmentation map 
$\si\colon  \Omega_\F(Q,\Phi^*V)^G\to \sD^{0,\bullet}(Q,V)$ 
given by the inclusion of invariant elements, and a vertical augmentation 
map $\sj\colon \sC(G,V)\to \sD^{\bullet,0}(Q,V)=\Gamma(\Phi_p^*(B_pG\times_M V))
$ given by  pullback under the  bundle projection $B_pG\times_M Q\to B_pG$.
For the total complex this gives  cochain maps
\[ \sC^\bullet(G,V)\stackrel{\sj}{\lra} \on{Tot}^\bullet(\sD(Q,V))\stackrel{\si}{\longleftarrow} \Omega^\bullet_\F(Q,\Phi^*V)^G.\]


\subsection{Differentiation}
Assuming that the $G$-action on $Q$ is proper, we may choose a  properly supported normalized Haar distribution $\mu=\{\mu_x\}$ for $G\ltimes Q\rra Q$. It determines a horizontal homotopy $\h$ on the double complex $\sD^{\bullet,\bullet}(Q,V)$; thus $[\delta,\h]=1-\si\circ \sp$, where
$\sp$ is the averaging map with respect to $\mu$. These are given by \eqref{eq:muhomotopy} and \eqref{eq:muprojection}, replacing $G$ with $G\ltimes Q$ and $V$ with $\Phi^*V$. Explicitly:
\begin{equation}\label{eq:hpsi}
 (\h\beta)(g_1,\ldots,g_{p-1})_x=(-1)^p\int_{\tz(a)=m} a\cdot \beta(g_1,\ldots,g_{p-1},a)_{a^{-1}\cdot x}\ \mu_x(a)\end{equation}
for $m=\sz(g_{p-1})$ and all $x\in \Phi^{-1}(m)$, and 
\[ (\sp\phi)_x=\int_{\tz(a)=\Phi(x)} a\cdot \phi_{a^{-1}\cdot x}\ \mu_x(a).\]
for $\phi\in \sD^{0,q}(Q,V)=\sC^q(T_\F Q,\Phi^*V)$. The Perturbation Lemma \ref{lem:perturbed} defines a homotopy inverse $\sp'=\sp\circ (1+\d\h)^{-1}$ to $\si$, and a  cochain map
\begin{equation}\label{eq:veq}
\VE_Q= \sp\circ (1+\d\h)^{-1}\circ \sj\colon \sC^\bullet(G,V)\to \Omega^\bullet_\F(Q,\Phi^*V)^G.
\end{equation}
%

\begin{example}
If $Q=G$ with the left-action of $G$, and using the Haar distribution from Example \ref{ex:actionhaar} \eqref{it:1}, we recover the homotopy operator $\h$ and projection $\sp$ from Section \ref{sec:vanest}. As we saw, this leads to 
the  van Est differentiation map $\VE_G$ of Weinstein-Xu. 
\end{example}

\begin{example}\label{ex:gkexample}
Let $G$ be a Lie group, $K$ a compact Lie subgroup, and $Q=G/K$. Let $\mu$ be the Haar distribution from Example \ref{ex:actionhaar} \eqref{it:2}, thus $\mu_{gK}(a)=\delta_K(g^{-1}a)$. Making a change of variables 
$a=gk$, we recover the homotopy operator $\h$ and projection $\sp$ from Section \ref{sec:classical}, leading to the differentiation map $\VE_{G/K}$ discussed there. 
\end{example}

\subsection{Integration}\label{subsec:integrationq}
Suppose that the submersion $\Phi\colon Q\to M$ admits a section $r\colon M\to Q$, i.e, $\Phi\circ r=\id_M$. Fixing $r$ we can think of $M$ as a submanifold as a submanifold of $Q$. 
Choose a tubular neighborhood embedding $\nu(Q,M)\to 
M$, taking the fibers of the normal bundle to the $\Phi$-fibers, 
to define a germ (along $[0,1]\times M$) of a retraction $\lambda\colon [0,1]\times Q\to Q$, with 
$\lambda_{t_1t_2}=\lambda_{t_1}\circ \lambda_{t_2}$, where 
\[ \lambda_0=r\circ \Phi,\ \ \lambda_1=\id_Q,\ \ \Phi\circ \lambda_t=\Phi,\ \ \lambda_t\circ r=r.\]  
In turn, it gives a homotopy operator $T$ on the localized foliated de Rham complex $\Omega_\F(Q,\Phi^*V)_M$.  

The discussion from Section \ref{subsec:integrationG} (for the case $Q=G$) extends to this setting  
in a straightforward fashion: One obtains a  homotopy operator $\k=(-1)^p T$ on the double complex $\sD(Q,V)_M$, with $\k\circ \k=0$ and $\sq\circ \k=0$; in turn, this defines a  homotopy inverse
$\sq'=\sq\circ (1+\delta \k)^{-1}\colon \on{Tot}(\sD(Q,V))\to \sC(G,V)$ to $\sj$, 
and the resulting van Est integration map 
\[ R_Q=\sq\circ (1+\delta \k)^{-1}\circ \si
\colon \ \Omega_\F(Q,\Phi^*V)^G\to \sC(G,V).\]
is described by the formula 
\[ R_Q(\alpha)(g_1\ldots,g_p)=\int_{[0,1]^p}\gamma^{(p)}(g_1,\ldots,g_p)^*\alpha,\]
for  $\alpha\in \Omega^p_\F(Q,\Phi^*V)^G$. Here
$ \gamma^{(p)}(g_1,\ldots,g_p)\colon [0,1]^p\to Q$
is defined  similar to \eqref{eq:gammaformula}: 
\begin{equation}\label{eq:gammaformula2}
  \gamma^{(p)}_{t_1,\ldots,t_p}(g_1,\ldots,g_p)=\lambda_{t_1}\Big(g_1\,\lambda_{t_2} \big(g_2\cdots 
 \lambda_{t_p}(g_p\cdot r(\sz(g_p)))\cdots\big)\Big).
 \end{equation}
If the $G$-action on $Q$ is furthermore proper, we also have the differentiation map $\VE_Q$, defined by the properly supported normalized Haar distribution $\mu$. 
In general, the van Est integration map $R_Q$ defined by $\lambda_t$ need not be a right inverse to the differentiation map $\on{VE}_Q$ -- the compatibility conditions of Lemma \ref{lem:backandforth} need not be satisfied, in general. One general setting where they are satisfied is the following. 

\begin{proposition}
Suppose $\mu=\{\mu_x\}$ is a properly supported normalized Haar distribution for $G\ltimes Q\rra Q$ 
with the property 
\[ \on{supp}(\mu)\subset G\times_M r(M)\]
(as a subset of $G\ltimes Q=G\times_M Q$). 
Then the conditions of Lemma \ref{lem:backandforth} are satisfied: that is,  $\h\circ \k=0,\ \sp\circ \k=0$. 
\end{proposition}
\begin{proof}
The condition $\on{supp}(\mu)\subset G\times_M r(M)$ is equivalent to the requirement that 
for all $x\in Q$, 
\begin{equation}
\label{eq:strangnecondition} \on{supp}(\mu_x)\subseteq\{a\in G|\ a^{-1}\cdot x\in r(M)\}.\end{equation}
But$(\k\beta)(g_1,\ldots,g_p)|_{r(M)}=0$, by the usual properties of the de Rham homotopy operator. On the other hand, 
the explicit formula \eqref{eq:hpsi} for the homotopy operator $\h$ shows that for all $a\in \on{supp}(\mu_x)$, 
\[ (\k\beta)(g_1,\ldots,g_p)_{a^{-1}\cdot x}=0
\Rightarrow 
(\h\k\beta)(g_1,\ldots,g_{p-1})\big|_x=0.\]  Hence $\h\circ \k=0$ if \eqref{eq:strangnecondition} holds true, and likewise $\sp\circ \k=0$. 
\end{proof}

This result `explains' Propositions \ref{prop:rightinversegk} and \ref{prop:rightinverse}: 

\begin{example}
Let $G$ be a Lie group, $K$ a compact subgroup, and $Q=G/K$.  The Haar distribution $\mu_{gK}=(l_g)_*\delta_K$ is supported in $gK\subset G$, which is the set of all $a\in G$ such that $a^{-1}gK=eK$, hence \eqref{eq:strangnecondition} holds true.
In fact, 
\[ \on{supp}(\mu)=G\times eK\subset G\ltimes G/K.\]
\end{example}

\begin{example}
Let $G\rra M$ be any Lie groupoid, and $Q=G$, with the Haar distribution $\mu_g=\delta_g$. Then \eqref{eq:strangnecondition} holds true, in fact, 
\[ \on{supp}(\mu)=G\times_M M\subset G\ltimes G.\]
\end{example}
\bigskip

\begin{appendix}
\section{Proof of Theorem \ref{th:vedifferentiation}}
Taking $Q=G$ in \eqref{eq:actions}, we have $p+1$ commuting $G$-actions on  $E_pG=B_pG\times_M G$; these commute with the principal action and descend to the actions on $B_pG$. The projection $\pi_p\colon E_pG\to M$ intertwines each of these actions with the trivial action on $M$; hence we obtain commuting $G$-actions on the vector bundles 
\begin{equation}\label{eq:isom1}
\pi_p^*(V\otimes \wedge^q A^*)=E_pG\times_M 
(V\otimes \wedge^q A^*)
\subset E_pG\times (V\otimes \wedge^q A^*),
\end{equation}
using the trivial action on the $V\otimes \wedge^q A^*$ factor. The infinitesimal action gives covariant derivatives $\nabla_\xi^{(i)}$ on 
$\sD^{p,q}(G,V)=\Gamma(\pi_p^*(\wedge^q A^*\otimes V))$; the derivatives for different $i$'s commute. They `lift' the operators $\nabla_\xi^{(i)}$ on $\sC^p(G,V)=\Gamma(B_pG\times_M V)$ introduced earlier.

\begin{lemma} \label{lem:helps}
\begin{enumerate}
	\item The maps $\sj\colon \sC^p(G,V)\to \sD^{p,0}(G,V)$ intertwine  $\nabla_\xi^{(i)}$ 
	for $i=0,\ldots,p$. 
\item The operators  $\nabla_\xi^{(i)}$ on the double complex 
commute with the vertical differential $\d$, and also with contractions $\iota_\zeta,\ \zeta\in \Gamma(A)$ and Lie derivatives $\L_\zeta$. 

\item The maps $\h=(-1)^p h_{p-1}^*\colon \sD^{p,q}(G,V)\to \sD^{p-1,q}(G,V)$ intertwine 
	$\nabla_\xi^{(i)}$ for $i=0,\ldots,p-1$, while 
	\begin{equation}\label{eq:lie2} 
	\ca{L}_\xi\circ \h =\h \circ (\nabla_\xi^{(p)}+\L_\xi).\end{equation}
	\end{enumerate}
\end{lemma}
\begin{proof}
\begin{enumerate}\item follows from the  equivariance of the map $\kappa_p$ with respect to the $i$-th action. 
\item Since $\iota_\zeta$ is equivariant for $i$-th action, it intertwines the operators $\nabla_\xi^{(i)}$.  Next, since $\kappa_p\colon E_pG\to B_pG$ is equivariant for the $i$-th action; the foliation $\F$ of $E_pG$ is preserved; i.e., the  infinitesimal action of $\Gamma(A)$ is by infinitesimal automorphisms of the Lie algebroid $T_\F E_pG$.  It follows that the action on $\sD^{p,\bullet}(G,V)$ preserves the differential $\d_{CE}$ and hence also $\d=(-1)^p\d_{CE}$. 
Finally, since $\L_\zeta=[\d_{CE},\iota_\zeta]$ it also intertwines the Lie derivatives (for the principal $G$-action); alternatively this follows directly because the $i$-th action commutes with the principal action.

\item 
The first part follows since the maps 
\[ h_{p-1}\colon E_{p-1}G\to E_pG,\ (g_1,\ldots,g_{p-1};g)\mapsto (g_1,\ldots,g_{p-1},g;\sz(g))\] 
(see \eqref{eq:hmap}) are equivariant for the actions labeled by $i=0,\ldots,p-1$. For \eqref{eq:lie2}, we need to consider both the generating vector fields $\xi^{(p)}$ for the $p$-th $G$-action and  the generators $\xi_{E_pG}\in\mf{X}(E_pG)$ of the principal action. In terms of $E_pG=B_pG\times_M G$,
	\[ \xi^{(p)}=(\xi^{L,p},-\xi^R),\ \ \ \xi_{E_pG}=(0,\xi^L),\ \ \]
where $\xi^{L,p}$ is the left-invariant vector field sitting on the last $G$-factor of $B_pG$. 
Since $\xi^L\sim_{u\circ \sz} \xi^L-\xi^R$ (where $u\colon M\to G$ is the inclusion of units), 
we see that 
	\[ \xi_{E_{p-1}G}\sim_{h_{p-1}} \xi_{E_pG}+\xi^{(p)},\]
	which implies 
Equation \eqref{eq:lie2}. 
\qedhere 
\end{enumerate}
\end{proof}

We are now in position to give the proof of Theorem \ref{th:vedifferentiation}.

\begin{proof}[Proof of Theorem \ref{th:vedifferentiation}]
	On elements of $\sC^p(G,V)=\Gamma(B_pG\times_M V)$, we have that 
	\[ \VE_G=(-1)^p \sp\circ (\d \h)^p\circ \sj.\] 
	Using  
	\begin{align*} \sj=\kappa_p^*&\colon\sC^p(G,V)\to \sD^{p,0}(G,V),\\
	\d \h=-\d_{CE}\circ h_{i-1}^*&\colon \sD^{i,p-i}(G,V)\to \sD^{i-1,p-i+1}(G,V),\\
	\sp=u^*&\colon \sD^{0,p}(G,V)\to \sC^p(A,V),\\ 
	\end{align*}
	this means that
	$\VE_G=
	u^*\circ \d_{CE}\circ h_0^*\circ \d_{CE} \circ \cdots \circ h_{p-1}^*\circ \kappa_p^*$. 
	Given $\xi_1,\ldots,\xi_p\in \Gamma(A)$ and $\sigma\in \Gamma(B_pG\times_M V)$, we want to compute
	\[  \VE_G(\sigma)(\xi_1,\ldots,\xi_p)=
	\iota_{\xi_p}\cdots \iota_{\xi_1} u^* \d_{CE}\, h_0^*\, \d_{CE}\, h_1^*\, \cdots \d_{CE}\,  h_{p-1}^*\, \kappa_p^*\, \sigma. \]
	Our strategy is to move the variables $\xi_p,\ldots,\xi_1$ to the right, while retaining their ordering (keeping $\xi_i$ to the left of $\xi_j$ if $i>j$). The commutators of contractions $\iota_\xi$ with $\d_{CE}$  produces Lie derivatives $\L_\xi=[\iota_\xi,\d_{CE}]$. Using Lemma \ref{lem:helps} and $\L_\xi\circ \kappa_p^*=0$,	 we find 
	\[ \L_\xi\circ h_{i-1}^*\cdots \d_{CE} \circ h_{p-1}^*\circ \kappa_p^*
	=h_{i-1}^*\circ \cdots \d_{CE} \circ h_{p-1}^*\circ \kappa_p^*  \circ \wh{\nabla}_\xi^{(i)} \]
	where we introduced the hat notation 
	\[ \wh{\nabla}_\xi^{(i)}=\nabla_\xi^{(i)}+\ldots+\nabla_\xi^{(1)},\]
	corresponding to the diagonal action for the actions labeled $1,\ldots,i$. (Note that the $0$-th action is not included.) We therefore obtain 
	\begin{align*} \VE_G(\sigma)(\xi_1,\ldots,\xi_p)&=u^*h_0^*\cdots h_{p-1}^*\kappa_p^*
	\sum_{s\in\mf{S}_p}\on{sign}(s) \wh{\nabla}_{\xi_p}^{(s(p))} \cdots \wh{\nabla}_{\xi_1}^{(s(1))}\sigma\\
	&=\big(\sum_{s\in\mf{S}_p}\on{sign}(s) \wh{\nabla}_{\xi_p}^{(s(p))} \cdots \wh{\nabla}_{\xi_1}^{(s(1))}\sigma\big)\Big|_M;
	\end{align*}
    here the second equality follows since the composition $\kappa_p\circ h_{p-1}\circ \cdots \circ h_0\circ u$ is just the inclusion 
    $M\to B_pG$. 
    To complete the proof, we argue that
    \begin{equation}\label{eq:sum}
    \sum_{s\in\mf{S}_p}\on{sign}(s) \wh{\nabla}_{\xi_p}^{(s(p))} \cdots \wh{\nabla}_{\xi_1}^{(s(1))}\end{equation}
    is equal to a similar sum with all hats removed.    
    Given $s\in\mf{S}_p$, let $i=s^{-1}(p)$. 
    Since
    \[ \wh{\nabla}_{\xi_i}^{(p)}=\nabla_{\xi_i}^{(p)}+\wh{\nabla}_{\xi_i}^{(p-1)}
    \]
    we see that the product
    \[ \wh{\nabla}_{\xi_p}^{(s(p))} \cdots (\wh{\nabla}_{\xi_i}^{(p)}-\nabla_{\xi_i}^{(p)})
    \cdots 
    \wh{\nabla}_{\xi_1}^{(s(1))}\]
    coincides with the corresponding expression for the permutation $s'$, given as the 
    composition of $s$ with the transposition of the indices $p,p-1$. Since $s,s'$ have opposite signs, 
it follows that \eqref{eq:sum} does not change when we remove the hats from 
all $\wh{\nabla}_{\xi_i}^{(s(i))}$ for which $s(i)=p$. 

Having done so, and assuming $p>2$, consider for a given $s\in \mf{S}_p$ the indices $i,j$ for which $s(i)=p,\ s(j)=p-1$.  (If $p=2$, we may simply put 
$\wh{\nabla}_\xi^{(1)}={\nabla}_\xi^{(1)}$, completing the proof.) An argument  similar to the first step shows that the expression 
 \begin{equation}\label{eq:ij} \wh{\nabla}_{\xi_p}^{(s(p))} \cdots \nabla_{\xi_i}^{(p)}
 \cdots 
 (\wh{\nabla}_{\xi_j}^{(p-1)}-\nabla_{\xi_i}^{(p-1)})
\cdots 
\wh{\nabla}_{\xi_1}^{(s(1))}\end{equation}
coincides with a similar expression for the composition of $s$ with 
transposition of the indices $p-1,p-2$. (We wrote \eqref{eq:ij} for the case that $i>j$; of course, if $i<j$ the $\nabla_{\xi_i}^{(p)}$ would appear to the right of 
$\wh{\nabla}_{\xi_j}^{(p-1)}-\nabla_{\xi_j}^{(p-1)}$.) Since those permutations have opposite signs, it shows that we may also remove the hat from the factors 
$\nabla_{\xi_j}^{(s(j)}$ with $s(j)=p-1$. Removing all the hats in this manner, we have proved the 
Weinstein-Xu formula  \[ \VE_G(\sigma)(\xi_1,\ldots,\xi_p)=\big(\sum_{s\in\mf{S}_p}\on{sign}(s) {\nabla}_{\xi_p}^{(s(p))} \cdots {\nabla}_{\xi_1}^{(s(1))}\sigma\big)\Big|_M.\]
\end{proof}

\end{appendix}

\bibliographystyle{amsplain}

\def\cprime{$'$} \def\polhk#1{\setbox0=\hbox{#1}{\ooalign{\hidewidth
			\lower1.5ex\hbox{`}\hidewidth\crcr\unhbox0}}} \def\cprime{$'$}
\def\cprime{$'$} \def\cprime{$'$} \def\cprime{$'$} \def\cprime{$'$}
\def\polhk#1{\setbox0=\hbox{#1}{\ooalign{\hidewidth
			\lower1.5ex\hbox{`}\hidewidth\crcr\unhbox0}}} \def\cprime{$'$}
\def\cprime{$'$} \def\cprime{$'$} \def\cprime{$'$} \def\cprime{$'$}
\providecommand{\bysame}{\leavevmode\hbox to3em{\hrulefill}\thinspace}
\providecommand{\MR}{\relax\ifhmode\unskip\space\fi MR }
\providecommand{\MRhref}[2]{%
	\href{http://www.ams.org/mathscinet-getitem?mr=#1}{#2}
}
\providecommand{\href}[2]{#2}

\end{document}